\documentclass[12pt,twoside]{amsart}
\usepackage{amssymb}
\usepackage{amscd}
\usepackage{xypic}

\title{Supplements to non-lc ideal sheaves}
\author{Osamu Fujino, Karl Schwede and Shunsuke Takagi}
\address{Department of Mathematics\endgraf Faculty of Science\endgraf
Kyoto University\endgraf
Kyoto 606-8502\endgraf Japan}
\email{fujino@math.kyoto-u.ac.jp}
\address{Department of Mathematics\endgraf The Pennsylvania State University\endgraf University Park, PA 16802\endgraf USA}
\email{schwede@math.psu.edu}
\address{Department of Mathematics\endgraf
Kyushu University\endgraf
Fukuoka 819-0395\endgraf
Japan}
\email{stakagi@math.kyushu-u.ac.jp}

\DeclareFontFamily{OMS}{rsfs}{\skewchar\font'60}
\DeclareFontShape{OMS}{rsfs}{m}{n}{<-5>rsfs5 <5-7>rsfs7 <7->rsfs10 }{}
\DeclareSymbolFont{rsfs}{OMS}{rsfs}{m}{n}
\DeclareSymbolFontAlphabet{\scr}{rsfs}

\newcommand{\Nlc}[0]{{\operatorname{Nlc}}}
\newcommand{\Nklt}[0]{{\operatorname{Nklt}}}
\newcommand{\Spec}[0]{{\operatorname{Spec\;}}}
\newcommand{\Exc}[0]{{\operatorname{Exc}}}
\newcommand{\Supp}[0]{{\operatorname{Supp\;}}}
\newcommand{\Bs}[0]{{\operatorname{Bs}}}
\newcommand{\mJ}{\mathcal{J}}
\newcommand{\I}{\scr I}
\newcommand{\ba}{\mathfrak{a}}
\newcommand{\tld}{\widetilde }
\newcommand{\DuBois}[1]{{\underline \Omega {}^0_{#1}}}

\newcommand{\myD}{{\bf D}}
\newcommand{\myH}{{\mathcal H}}
\newcommand{\myR}{{R}}
\newcommand{\mydot}{{{\,\begin{picture}(1,1)(-1,-2)\circle*{2}\end{picture}\ }}}
\newcommand{\cO}{\mathcal O}
\newcommand{\qis}{\simeq_{\text{qis}}}
\newcommand{\tensor}{\otimes}
\newcommand{\sH}{\scr{H}}
\DeclareMathOperator{\red}{red}
\DeclareMathOperator{\coherent}{{coh}}
\DeclareMathOperator{\sHom}{{\sH}om}
\DeclareMathOperator{\Sing}{{Sing}}
\DeclareMathOperator{\Ann}{{Ann}}

\newcommand{\ib}{\mathfrak{b}}
\newcommand{\Q}{\mathbb{Q}}
\newcommand{\R}{\mathbb{R}}
\newcommand{\C}{\mathbb{C}}
\newcommand{\Z}{\mathbb{Z}}
\newcommand{\N}{\mathbb{N}}
\newcommand{\F}{\mathbb{F}}
\newcommand{\m}{\mathfrak{m}}
\DeclareMathOperator{\Div}{{div}}
\DeclareMathOperator{\Hom}{{Hom}}

\newtheorem{thm}{Theorem}[section]
\newtheorem{lem}[thm]{Lemma}
\newtheorem{cor}[thm]{Corollary}
\newtheorem{prop}[thm]{Proposition}
\newtheorem*{claim}{Claim}

\theoremstyle{definition}
\newtheorem{ex}[thm]{Example}
\newtheorem{defn}[thm]{Definition}
\newtheorem{propdef}[thm]{Proposition-Definition}
\newtheorem{rem}[thm]{Remark}
\newtheorem*{ack}{Acknowledgments}
\newtheorem*{notation}{Notation}
\newtheorem{que}[thm]{Question}
\newtheorem{conj}[thm]{Conjecture}

\newtheorem{say}[thm]{}
\RequirePackage[dvipsnames,usenames]{color}

\begin{document}
\bibliographystyle{amsalpha+}
\maketitle

\begin{abstract}
We consider various definitions of non-lc ideal sheaves -- generalizations of the multiplier ideal sheaf which
define the non-lc (non-log canonical) locus.
We introduce the maximal non-lc ideal sheaf and intermediate non-lc ideal sheaves
and consider the restriction theorem for these ideal sheaves.
We also begin the development of the theory of a characteristic $p > 0$ analog of maximal non-lc ideals, utilizing some recent work of Blickle.
\end{abstract}

\tableofcontents

\section{Motivation}\label{sec-mot}
In this short section, we explain our motivation for the study of
{\em{non-lc ideal sheaves}}.

\begin{say}[Motivation]
Let $X$ be a normal variety and $\Delta$ an effective $\mathbb Q$-divisor on $X$
such that $K_X+\Delta$ is $\mathbb Q$-Cartier.
In this situation, we want to define an ideal sheaf $I(X, \Delta)$ satisfying
the following properties.
\begin{itemize}
\item[(A)] The pair $(X, \Delta)$ is log canonical
if and only if $I(X, \Delta)=\cO_X$.
\item[(B)] (Kodaira type vanishing theorem). Assume that
$X$ is projective.
Let $D$ be a Cartier divisor on $X$ such that $D-(K_X+\Delta)$ is ample. Then
$$
H^i(X, I(X, \Delta)\otimes \cO_X(D))=0
$$
for every $i>0$.
\item[(C)] (Bertini type theorem).
Let $H$ be a general member of a free linear system $\Lambda$ on $X$. Then
$$
I(X, \Delta)=I(X, \Delta+H).
$$
\item[(D)] (Restriction theorem).
Assume that $\Delta=S+B$ such that
$S$ is a normal prime Weil divisor on $X$, $B$ is an effective
$\mathbb Q$-divisor, and that $S$ and $B$
have no common irreducible components.
Then
$$
I(X, S+B)|_S=I(S, B_S),
$$
where $(K_X+S+B)|_S=K_S+B_S$.
\end{itemize}
\end{say}
We have already known that the ({\em{minimal}}) {\em{non-lc ideal sheaf}}
$\mathcal J_{NLC}(X, \Delta)$ introduced in \cite{fujino-non} satisfies
all the above properties.
The {\em{intermediate non-lc ideal sheaves}}
$\mathcal J'_{l}(X, \Delta)$ for every negative integer $l$, which will be defined in Section
\ref{sec3.5} below, satisfy (A), (B), and (C). However,
in general, (D) does not always hold for $\mathcal J'_l (X, \Delta)$ with
$l=-1, -2, \cdots$.
The {\em{maximal non-lc ideal sheaf}}
$\mathcal J'(X, \Delta)$ also satisfies (A), (B), and (C),
and we do not know if (D) holds for $\mathcal J'(X, \Delta)$ or not.  However, we have some evidence that it is true.
In particular, we will give partial answers to this question in Section \ref{sec3} and in Section \ref{KarlSection}.  In section \ref{KarlSection}, we also mention a link between $\mathcal{J}'(X, \Delta)$ and ideals that appear naturally in the study of the Hodge-theory of singular varieties.  We conclude this paper by developing a characteristic $p > 0$ analog of $\mathcal{J}'(X, \Delta)$, relying heavily on some recent interesting work of Blickle, see \cite{Bl}.

Finally, it should be noted that, in various presentations, S\'andor Kov\'acs has recently been discussing how the ideal $\mathcal{J}'(X, \Delta)$ is a natural ideal to consider.  His work in this direction is independent of the authors although certainly inspired by connections with Du Bois singularities; see Section \ref{KarlSection}.

\section{Introduction}
The main purpose of
this paper is to consider variants of the {\em{non-lc ideal
sheaf}} $\mathcal J_{NLC}$ introduced in \cite{fujino-non}.
We also consider various {\em{non-klt ideal
sheaves}}. We will explain our motivation, observations, and some attempts in the
study of {\em{non-lc ideal sheaves}}.

Let $D$ be an effective $\mathbb R$-divisor on a smooth complex
variety $X$.
We put
$$
\mathcal J'(X, D):=\mathcal J(X, (1-\varepsilon)D)
$$
for $0<\varepsilon \ll 1$,
where the right hand side is the {\em{multiplier ideal sheaf}}
associated to $(1-\varepsilon)D$ and it is
independent of $\varepsilon$ for sufficiently small $0<\varepsilon \ll 1$.
By the definition of $\mathcal J'(X, D)$,
the pair $(X, D)$ is log canonical
if and only if $\mathcal J'(X, D)=\cO_X$.
We call $\mathcal J'(X, D)$ the {\em{maximal non-lc ideal sheaf}}
of $(X, D)$. We will discuss the definition and
the basic properties of $\mathcal J'$ in Section \ref{sec2}.
In general, $\mathcal J(X, D)\subsetneq \mathcal J(X, (1-\varepsilon)D)$
for $0<\varepsilon \ll 1$ and the relationship between $\mathcal J(X, D)$
and $\mathcal J'(X, D)$ is not clear.
So, we need new ideas and techniques to handle $\mathcal J'(X, D)$.
We believe that
the Kawamata--Viehweg--Nadel vanishing theorem is not powerful enough
for the study of $\mathcal J'(X, D)$.  However, the new cohomological
package of the first author, explained in Section \ref{sec4}, seems well suited for this task.

Let $X$ be a smooth variety and $S$ a smooth
irreducible divisor on $X$.
Let $B$ be an effective $\mathbb R$-divisor on $X$ such that
$S\not\subset \Supp B$.
Then we have the following equality
$$
\mathcal J'(S, B|_S)=\mathcal J'(X, S+B)|_S.
$$
See, for example, Theorem \ref{re-th}.
We call it the {\em{restriction theorem}}.
We will partially generalize it to the case of singular varieties in Section \ref{KarlSection}.
In \cite{fujino-non}, the first author introduced the notion
of ({\em{minimal}}) {\em{non-lc ideal sheaves}} $\mathcal J_{NLC}$
and proved the restriction theorem
$$
\mathcal J_{NLC}(S, B|_S)=\mathcal J_{NLC}(X, S+B)|_S.
$$

Both of $\mathcal J'(X, D)$ and $\mathcal J_{NLC}(X, D)$ define the non-lc
locus of the pair $(X, D)$. However,
in general, $\mathcal J'(X, D)$ does not always coincide with
$\mathcal J_{NLC}(X, D)$.
We note that
$$
\mathcal J(X, D)\subset \mathcal J_{NLC}(X, D)\subset
\mathcal J'(X, D)
$$
holds by the definitions of
$\mathcal J$, $\mathcal J_{NLC}$, and $\mathcal J'$.
Although $\mathcal J_{NLC}(X, D)$ seems to be
the most natural ideal that defines the non-lc locus
of $(X, D)$
from the point of view of the minimal model program (cf.~\cite{fundamental}),
$\mathcal J'(X, D)$ may be more suitable to
the theory of multiplier ideal
sheaves than $\mathcal J_{NLC}(X, D)$.

More generally, we consider a family of non-lc
ideal sheaves.  We define {\em{intermediate non-lc ideal
sheaves}} $\mathcal J'_l(X, D)$ for
every negative integer $l$. By the definition of
$\mathcal J'_l(X, D)$ (which is a sheaf that varies with each negative integer $l$),
$\mathcal J'_l(X, D)$ defines the non-lc locus of $(X, D)$ and satisfies many
of the same useful properties that the first author's original non-lc ideal $\mJ_{NLC}(X, D)$ enjoys.

Furthermore, there are natural inclusions (where again, the $l$ vary of the \emph{negative} integers)
\begin{align*}
\mathcal J_{NLC}(X, D)&\subset\cdots \subset \mathcal J'_{l-1}(X, D)\\ &\subset \mathcal J'_l(X, D)\subset
\mathcal J'_{l+1}(X, D)\subset \cdots \subset \mathcal J'(X, D)\subset \cO_X.
\end{align*}
Similarly, we also define a family of {\em{non-klt ideal sheaves}} $\mathcal J_l (X, D)$ for
every non-positive integer $l$.  These sheaves satisfy
\begin{align*}
\mathcal J(X, D)&\subset\cdots \subset \mathcal J_{l-1}(X, D)\\ &\subset \mathcal J_l(X, D)\subset
\mathcal J_{l+1}(X, D)\subset \cdots \subset \mathcal J_0(X, D)\subset \cO_X
\end{align*}
and put possibly different scheme structures on the non-klt locus of $(X, D)$.
We have natural inclusions
$$
\mathcal J_l(X, D)\subset \mathcal J'_l(X, D)
$$
for every negative integer $l$, as well as the inclusions
\begin{align*}
\mathcal J(X, D)\subset \mathcal J_{NLC}(X, D) \quad \text{and}\quad
\mathcal J_0(X, D)\subset \mathcal J'(X, D).
\end{align*}
Let $W$ be the union of all the lc centers of $(X, D)$ (see our slightly non-standard definition of
lc centers in Section \ref{sec33}) and $U=X\setminus W$.
Then
$$
\mathcal J_l(X, D)|_U=\mathcal J'_l(X, D)|_U
$$
for every negative integer $l$,
\begin{align*}
\mathcal J(X, D)|_U=\mathcal J_{NLC}(X, D)|_U \quad \text{and}
\quad \mathcal J_0(X, D)|_U=\mathcal J'(X, D)|_U.
\end{align*}
Because the multiplier ideal sheaf has emerged as such a fundamental tool in
higher dimensional algebraic geometry, it is natural to desire a non-lc
ideal sheaf which agrees with the multiplier ideal sheaf in as wide a setting as possible.
If we assume that this is a desired property, then $\mathcal J_{NLC}(X, D)$
is the right generalization of $\mathcal J(X, D)$.

On the other hand, with regards to condition (B) from Section \ref{sec-mot},
one way to interpret the term $I(X, \Delta)$ is as a correction term which mitigates
for the singularities of $\Delta$.  From this point of view, using the maximal non-lc ideal
sheaf $\mathcal J'(X, \Delta)$ gives a heuristically stronger statement as it says one has
to ``adjust'' to lesser extent.

The multiplier ideal $\mathcal{J}(X = \Spec R, \Delta)$ is also very closely related to the test ideal $\tau_b(R, \Delta)$, a notion that appears in the theory of commutative algebra in positive characteristic, see for example \cite{Ta}.  We conclude this paper with several sections which explore a positive characteristic analog of $\mathcal{J}'(X, \Delta)$ which we call the non-$F$-pure ideal and denote it by $\sigma(X, \Delta)$.  In order to define this ideal, we rely heavily on some recent work of Blickle, see \cite{Bl}.  In section \ref{secNonFpureVsNonLC} we then relate the characteristic zero notion $\mJ'(X, \Delta)$ and the characteristic $p > 0$ notion $\sigma(R, \Delta)$.
In the final section, we prove a restriction theorem for $\sigma(X, \Delta)$, including a proof that the formation of $\sigma(X, \Delta)$ commutes with the restriction to an arbitrary codimension normal $F$-pure center (which is a characteristic $p > 0$ analog of a log canonical center).

We summarize the contents of this paper.
This paper is divided into two parts.
Part I, consisting of Section \ref{sec33}--\ref{KarlSection},  is devoted to the study of variants of non-lc ideal sheaves $\mJ_{NLC}$ on complex algebraic varieties.
Part II, consisting of Section \ref{secNonFpureIdeals}--\ref{secResThmForNonFpure}, is devoted to the study of a positive characteristic analog of the maximal non-lc ideal sheaves.
These two parts are independent to each other, except for the definition of the maximal non-lc ideal sheaves.

In Section \ref{sec33}, we define {\em{lc centers}},
{\em{non-klt centers}},
and {\em{non-lc centers}}. It is very important to distinguish
these three notions.
In Section \ref{sec4}, we recall
Ambro's formulation of Koll\'ar's torsion-free and vanishing theorems for the reader's convenience.
In Section \ref{sec3-non}, we recall the notion of non-lc ideal
sheaves introduced in \cite{fujino-non}.
In Section \ref{sec1.5}, we discuss how to
define certain
non-lc ideal sheaves. We also discuss some properties
which should be satisfied by these
ideal sheaves. This section is an informal discussion.
In Section \ref{sec2}, we will define the {\em{maximal
non-lc ideal sheaf}} $\mathcal J'$ and
investigate basic properties of $\mathcal J'$.
In Section \ref{sec3.5}, we introduce the notion of {\em{intermediate
non-lc ideal sheaves}}.
Section \ref{sec7} is a supplement to the fundamental
theorems for the log minimal model
program in \cite{fundamental}.
In Section \ref{sec-klt},
we discuss various {\em{non-klt ideal sheaves}}.
In Section \ref{sec-diff}, we recall Shokurov's {\em{differents}} for
the restriction theorem discussed in Section \ref{sec3}.
Sections \ref{sec3} and \ref{KarlSection} are attempts to
prove the restriction theorem for $\mathcal J'$.  Also in section \ref{KarlSection}, we explain how $\mathcal{J}'(X, \Delta)$ appears in the study of the Hodge theory of singular varieties.
In Section \ref{secNonFpureIdeals}, we introduce a characteristic $p$ analog of the maximal non-lc-ideal $\mathcal{J}'(X, \Delta)$, called a non-F-pure ideal, and investigate its basic properties.
In Section \ref{secNonFpureVsNonLC}, we explore the relationship between non-F-pure ideals and maximal non-lc ideals, which is followed by Section \ref{secResThmForNonFpure} where we prove a restriction theorem for non-F-pure ideals.

We will work over the complex number field $\mathbb C$ throughout Part I.
But we note that by using the Lefschetz principle, we can
extend everything to the case where
the base field is an algebraically closed field of characteristic zero.
Also, we will use the following notation freely.

\begin{notation}
(i) For an $\mathbb R$-Weil divisor
$D=\sum _{j=1}^r d_j D_j$ such that
$D_j$ is a prime divisor for every $j$ and
$D_i\ne D_j$ for $i\ne j$, we define
the {\em{round-up}} $\ulcorner D\urcorner =\sum _{j=1}^{r}
\ulcorner d_j\urcorner D_j$
(resp.~the {\em{round-down}} $\llcorner D\lrcorner
=\sum _{j=1}^{r} \llcorner d_j \lrcorner D_j$),
where for every real number $x$,
$\ulcorner x\urcorner$ (resp.~$\llcorner x\lrcorner$) is
the integer defined by $x\leq \ulcorner x\urcorner <x+1$
(resp.~$x-1<\llcorner x\lrcorner \leq x$).
The {\em{fractional part}} $\{D\}$
of $D$ denotes $D-\llcorner D\lrcorner$.
We define
\begin{align*}&
D^{=k}=\sum _{d_j=k}d_jD_j=k\sum _{d_j=k}D_j,
\ \ D^{\leq k}=\sum_{d_j\leq k}d_j D_j, \\ &
D^{<k}=\sum_{d_j< k}d_j D_j, \ \ D^{\geq k}=\sum_{d_j\geq k}d_j D_j
\ \ \text{and}\ \ \
D^{>k}=\sum_{d_j>k}d_j D_j
\end{align*} for
every $k\in \mathbb R$.
We put
$$
{}^k\!D=\Supp D^{=k}.
$$
We note that ${}^0\!D=\Supp D^{=0}=0$ and ${}^1\!D=\Supp D^{=1}=D^{=1}$.
We call $D$ a {\em{boundary}}
$\mathbb R$-divisor if
$0\leq d_j\leq 1$
for every $j$.
We note that $\sim _{\mathbb Q}$ (resp.~$\sim _{\mathbb R}$)
denotes the $\mathbb Q$-linear (resp.~$\mathbb R$-linear) equivalence
of $\mathbb Q$-divisors (resp.~$\mathbb R$-divisors).

(ii) For a proper birational morphism $f:X\to Y$,
the {\em{exceptional locus}} $\Exc (f)\subset X$ is the locus where
$f$ is not an isomorphism.
\end{notation}

\begin{ack}
The first author was partially supported by the Grant-in-Aid for Young Scientists (A) $\sharp$20684001 from JSPS and by the Inamori Foundation.
The second author was partially supported by an NSF postdoctoral fellowship and by the NSF grant DMS-1064485. 
The third author was partially supported by the Grant-in-Aid for Young Scientists (B) $\sharp$20740019 from JSPS and by the Program for Improvement of Research Environment for
Young Researchers from SCF commissioned by MEXT of Japan.
They would like the referee for many useful suggestions. 
A part of this work was done during visits of all the authors to MSRI.
They are grateful to MSRI for its hospitality and support.
\end{ack}

\part{Variants of non-lc ideals $\mJ_{NLC}$}
This part is devoted to the study of variants of non-lc ideal sheaves $\mJ_{NLC}$ on complex algebraic varieties.
\section{Lc centers, non-klt centers, and non-lc centers}\label{sec33}

In this section, we quickly recall the notion of lc and klt paris and
define {\em{lc centers}}, {\em{non-klt centers}}, and
{\em{non-lc centers}}.

\begin{say}[Discrepancies, lc and klt pairs, etc.]Let $X$ be a normal variety and $B$
an effective $\mathbb R$-divisor
on $X$ such that $K_X+B$ is $\mathbb R$-Cartier.
Let $f:Y\to X$ be a resolution such that
$\Exc (f)\cup f^{-1}_*B$ has a simple normal crossing
support, where $f^{-1}_*B$ is the strict transform of $B$ on $Y$.
We write $$K_Y=f^*(K_X+B)+\sum _i a_i E_i$$ and
$a(E_i, X, B)=a_i$.
We say that
$(X, B)$ is
$$
\begin{cases}
{\text{log canonical (lc, for short)}} &{\text{if $a_i\geq -1$ for every $i$, and}}, \\
{\text{kawamata log terminal (klt, for short)}} &{\text{if $a_i>-1$ for every $i$}}.
\end{cases}
$$
Note that the {\em{discrepancy}} $a(E, X, B)\in \mathbb R$ can be
defined for every prime divisor $E$ {\em{over}} $X$.
By definition, there exists the largest Zariski open set
$U$ (resp.~$U'$) of $X$ such that $(X, B)$
is lc (resp.~klt) on $U$ (resp.~$U'$).
We put $\Nlc(X, B)=X\setminus U$ (resp.~$\Nklt(X, B)=X\setminus
U'$) and call it
the {\em{non-lc locus}} (resp.~{\em{non-klt locus}})
of the pair
$(X, B)$.
We sometimes simply denote $\Nlc (X, B)$ by $X_{NLC}$.
We will discuss various scheme structures on $\Nlc(X, B)$ (resp.~$\Nklt(X, B)$)
in Section \ref{sec7} (resp.~in Section \ref{sec-klt}).

Let $E$ be a prime divisor {\em{over}} $X$. The closure
of the image of $E$ on $X$ is denoted by
$c_X(E)$ and
called the {\em{center}} of $E$ on $X$.
\end{say}

\begin{say}[Lc centers, non-klt centers, and non-lc centers]
Let $X$ be a normal variety and $B$ an effective $\mathbb R$-divisor
on $X$ such that $K_X+B$ is $\mathbb R$-Cartier.
Let $E$ be a prime divisor over $X$.
In this paper, we use the following terminology.
The center $c_X(E)$ is
$$
\begin{cases}
\text{an lc center} & \text{if $a(E, X, B)=-1$ and $c_X(E)\not\subset \Nlc (X, B)$, }\\
\text{a non-klt center} &\text{if $a(E, X, B)\leq -1$, and}\\
\text{a non-lc center} &\text{if $a(E, X, B)<-1$.}
\end{cases}
$$
The above terminology is slightly different from the usual one.
We note that it is very important to distinguish
lc centers, non-klt centers, and non-lc centers in our theory.
In the traditional
theory of multiplier ideal sheaves, we can not distinguish among lc centers,
non-klt centers, and non-lc centers.
In our new framework, the notion of lc centers plays very important roles. It is because our arguments heavily depend
on the new cohomological package reviewed in Section
\ref{sec4}.
It is much more powerful than the Kawamata--Viehweg--Nadel
vanishing theorem.
We note that an lc center is a non-klt center.
\end{say}

The next lemma is almost obvious by the definition of lc centers.

\begin{lem}
The number of lc centers of $(X, B)$ is finite even if $(X, B)$ is not log canonical.
\end{lem}

We note the following elementary example.

\begin{ex}
Let $X=\mathbb C^2=\Spec \mathbb C[x, y]$ and
$C=(y^2=x^3)$.
We consider the pair $(X, C)$.
Then we can easily check that there is a prime divisor
$E$ over $X$ such that $a(E, X, C)=-1$ and
$c_X(E)$ is the origin $(0, 0)$ of $\mathbb C^2$ and
that $(X, C)$ is not lc at $(0, 0)$.
Therefore, the center $c_X(E)$ is a non-klt center but not an lc center of $(X, C)$.
\end{ex}

\section{New cohomological package}\label{sec4}

We quickly review Ambro's formulation
of torsion-free and vanishing theorems in a simplified form.
For more advanced topics and the proof, see \cite[Chapter 2]{fuji-book}.
The paper \cite{fuji0} may help the reader to understand the
proof of Theorem \ref{ap1}.
We think that it is not so easy to
grasp the importance of Theorem \ref{ap1}.
We recommend the reader to learn how to use
Theorem \ref{ap1} in \cite{fujino-non},
\cite{fuji-book},
\cite{fundamental}, and this paper.

\begin{say}[Global embedded simple normal crossing pairs]
Let $Y$ be a simple normal crossing divisor
on a smooth
variety $M$ and $D$ an $\mathbb R$-divisor
on $M$ such that
$\mathrm{Supp}(D+Y)$ is simple normal crossing and that
$D$ and $Y$ have no common irreducible components.
We put $B=D|_Y$ and consider the pair $(Y, B)$.
Let $\nu:Y^{\nu}\to Y$ be the normalization.
We put $K_{Y^\nu}+\Theta=\nu^*(K_Y+B)$.
A {\em{stratum}} of $(Y, B)$ is an irreducible component of $Y$ or
the image of some lc center of $(Y^\nu, \Theta^{=1})$.

When $Y$ is smooth and $B$ is an $\mathbb R$-divisor
on $Y$ such that
$\Supp B$ is simple normal crossing, we
put $M=Y\times \mathbb A^1$ and $D=B\times \mathbb A^1$.
Then $(Y, B)\simeq (Y\times \{0\}, B\times \{0\})$ satisfies
the above conditions.
\end{say}

\begin{thm}\label{ap1}
Let $(Y, B)$ be as above.
Assume that $B$ is a boundary $\mathbb R$-divisor.
Let
$f:Y\to X$ be a proper morphism and $L$ a Cartier
divisor on $Y$.

$(1)$ Assume that $L-(K_Y+B)$ is $f$-semi-ample.
Let $q$ be an arbitrary non-negative integer.
Then
every non-zero local section of $R^qf_*\cO_Y(L)$ contains
in its support the $f$-image of
some stratum of $(Y, B)$.

$(2)$ Let $\pi:X\to V$ be a proper morphism and
assume that $L-(K_Y+B) \sim _{\mathbb R}f^*H$ for
some $\pi$-ample $\mathbb R$-Cartier
$\mathbb R$-divisor $H$ on $X$.
Then, $R^qf_*\cO_Y(L)$ is $\pi_*$-acyclic, that is,
$R^p\pi_*R^qf_*\cO_Y(L)=0$ for every $p>0$ and $q\geq 0$.
\end{thm}

\begin{rem}
It is obvious that
the statement of Theorem \ref{ap1} (1) is
equivalent to the following one.

$(1^{\prime})$ Assume that $L-(K_Y+B)$ is $f$-semi-ample.
Let $q$ be an arbitrary non-negative integer.
Then
every associated prime of $R^qf_*\cO_Y(L)$
is the generic point of the $f$-image of some stratum
of $(Y, B)$.
\end{rem}

For the proof of Theorem \ref{ap1}, see \cite[Theorem 2.39]{fuji-book}.

\begin{rem}\label{rem24} In Theorem \ref{ap1} (2),
it is sufficient to assume that $H$ is $\pi$-nef and
$\pi$-log big.
See \cite[Theorem 2.47]{fuji-book}.
We omit the technical details on nef and log big divisors
in order to keep this paper readable.
\end{rem}

\section{Non-lc ideal sheaves}\label{sec3-non}
Let us recall the definition of non-lc ideal sheaves
(cf.~\cite[Section 2]{fujino-non} and
\cite[Section 7]{fundamental}).

\begin{defn}[Non-lc ideal sheaf]\label{def61}
Let $X$ be a normal variety and $B$
an $\mathbb R$-divisor on $X$ such that
$K_X+B$ is $\mathbb R$-Cartier.
Let $f:Y\to X$ be a resolution with $K_Y+B_Y=f^*(K_X+B)$
such that
$\Supp B_Y$ is simple normal crossing.
Then we put
\begin{align*}
\mathcal J_{NLC}(X, B)
&=f_*\cO_Y(\ulcorner -(B_Y^{<1})\urcorner-
\llcorner B_Y^{>1}\lrcorner)\\& =f_*\cO_Y
(-\llcorner B_Y\lrcorner
+B^{=1}_Y)
\end{align*} and
call it the ({\em{minimal}}) {\em{non-lc ideal sheaf associated to $(X, B)$}}.
If $B$ is effective, then
$\mathcal J_{NLC}(X, B)\subset \cO_X$.
\end{defn}

The ideal sheaf $\mathcal J_{NLC}(X, B)$ is independent of the choice of resolution, and thus well-defined,
by the following
easy lemma.

\begin{lem}\label{lem62}
Let $g:Z\to Y$ be a proper birational morphism between
smooth varieties and $B_Y$ an
$\mathbb  R$-divisor on $Y$ such
that $\Supp B_Y$ is simple normal crossing.
Assume that $K_Z+B_Z=g^*(K_Y+B_Y)$ and
that $\Supp B_Z$ is simple normal crossing.
Then we have $$g_*\cO_Z(\ulcorner -(B^{<1}_Z)\urcorner
-\llcorner B^{>1}_Z\lrcorner)\simeq
\cO_Y(\ulcorner -(B^{<1}_Y)\urcorner
-\llcorner B^{>1}_Y\lrcorner). $$
\end{lem}
\begin{proof}
By $K_Z+B_Z=g^*(K_Y+B_Y)$, we
obtain
\begin{align*}
K_Z=&g^*(K_Y+B^{=1}_Y+\{B_Y\})\\&+g^*
(\llcorner B^{<1}_Y\lrcorner+\llcorner B^{>1}_Y\lrcorner)
-(\llcorner B^{<1}_Z\lrcorner+\llcorner B^{>1}_Z\lrcorner)
-B^{=1}_Z-\{B_Z\}.
\end{align*}
If $a(\nu, Y, B^{=1}_Y+\{B_Y\})=-1$ for a prime divisor
$\nu$ over $Y$, then
we can check that $a(\nu, Y, B_Y)=-1$ by using
\cite[Lemma 2.45]{km}.
Since $g^*
(\llcorner B^{<1}_Y\lrcorner+\llcorner B^{>1}_Y\lrcorner)
-(\llcorner B^{<1}_Z\lrcorner+\llcorner B^{>1}_Z\lrcorner)$ is
Cartier, we can easily see that
$$g^*(\llcorner B^{<1}_Y\lrcorner+\llcorner B^{>1}_Y\lrcorner)
=\llcorner B^{<1}_Z\lrcorner+\llcorner B^{>1}_Z\lrcorner+E, $$
where $E$ is an effective $f$-exceptional Cartier divisor.
Thus, we obtain
$$g_*\cO_Z(\ulcorner -(B^{<1}_Z)\urcorner
-\llcorner B^{>1}_Z\lrcorner)\simeq
\cO_Y(\ulcorner -(B^{<1}_Y)\urcorner
-\llcorner B^{>1}_Y\lrcorner).$$
This completes the proof.
\end{proof}

The next lemma is obvious by definition:~Definition \ref{def61}.

\begin{lem}
Let $X$ be a normal variety and $B$ an effective $\mathbb R$-divisor on $X$ such that
$K_X+B$ is $\mathbb R$-Cartier. Then $(X, B)$ is lc if and only if
$\mathcal J_{NLC}(X, B)=\cO_X$.
\end{lem}

In the following sections, we consider variants of non-lc ideal sheaves.

\section{Observations towards non-lc ideal sheaves}\label{sec1.5}
First, we informally define $\mathcal J'$ as a limit
of multiplier ideal sheaves. We will call $\mathcal J'(X, B)$
the {\em{maximal non-lc ideal sheaf}} of the pair $(X, B)$.
For the details, see Section \ref{sec2}.

\begin{say}\label{maxnonlcdefsm}
Let $D$ be an effective $\mathbb R$-divisor on a smooth
variety $X$. Let $f:Y\to X$ be a resolution
such that $\Exc (f)\cup \Supp f^{-1}_*D$ is simple normal crossing.
Then the {\em{multiplier ideal sheaf}} $\mathcal J(X, D)\subset \cO_X$
associated to
$D$ was defined to
be
$$
\mathcal J(X, D)=f_*\cO_Y(K_{Y/X}-\llcorner f^*D\lrcorner),
$$
where $K_{Y/X}=K_Y-f^*K_X$.
In this situation,
we put
$$
\mathcal J'(X, D)=\mathcal J(X, (1-\varepsilon)D)
$$
for $0<\varepsilon \ll 1$.
We note that the right hand side is independent of
$\varepsilon$ for
$0<\varepsilon \ll 1$.
Therefore, we can write
$$
\mathcal J'(X, D)=\bigcap _{0<\varepsilon}\mathcal J(X, (1-\varepsilon)D)
$$
since
$$
\mathcal J(X, (1-\varepsilon)D)\subset \mathcal J(X, (1-\varepsilon')D)
$$
for $0<\varepsilon <\varepsilon '$.
We write $K_Y+\Delta_Y=f^*(K_X+D)$.
Then
$$
\mathcal J(X, D)=f_*\cO_Y(-\llcorner \Delta_Y\lrcorner),
$$
and
$$
\mathcal J(X, (1-\varepsilon)D)=f_*\cO_Y(-\llcorner
\Delta_Y\lrcorner+\sum _{k=-\infty}^{\infty}{}^k\!\Delta_Y)
$$
for $0<\varepsilon \ll 1$.
Since ${}^k\!\Delta_Y$ is $f$-exceptional for $k<0$,
we can write
$$
\mathcal J'(X, D)=f_*\cO_Y(-\llcorner
\Delta_Y\lrcorner +\sum _{k=1}^{\infty}{}^k\!\Delta_Y),
$$
This expression is very useful for generalizations.

By definition, we can easily check that
$$\mathcal J(X, (1+\varepsilon)D)=\mathcal J(X, D)$$ for
$0<\varepsilon \ll1$ and that
$\mathcal J(X, (1-\varepsilon) D)=\mathcal J(X, D)$ for
$0<\varepsilon \ll 1$ if and only if
$t=1$ is not a jumping number of $\mathcal J(X, tD)$.
In this paper, we are mainly interested in the case when $D$ is a reduced divisor.
In this case, $t=1$ is a jumping number of $\mathcal J(X, tD)$ and then
$\mathcal J'(X, D)\supsetneq \mathcal J(X, D)$.
\end{say}

Next, we observe various properties which should
be satisfied by {\em{non-lc ideal sheaves}}.

\begin{say}
Let $X$ be a smooth projective variety and
$B$ an effective integral Cartier divisor on $X$ such that
$\Supp B$ is simple normal crossing.
We can write $B=\sum _{k=1}^{\infty} kB_k$,
where $B_k:={}^k\!B=\Supp B^{=k}$.
We would like to define an ideal sheaf
$I(X, B)\subset \cO_X$ such that
$\Supp \cO_X/I(X, B)=\Nlc (X, B)$.
Let us put
$$I(X, B)=\cO_X(-\sum _{k=2}^{\infty}m_kB_k)$$ for some
$m_k\geq 1$ for every $k\geq 2$. Then
$I(X, B)$ defines the non-lc locus of the pair
$(X, B)$.
Let $L$ be a Cartier divisor on $X$ such that
$A:=L-(K_X+B)$ is ample.
For various geometric applications, we think that
it is natural to require
$$
H^i(X, \cO_X(L)\otimes I(X, B))=0
$$ for
all $i>0$.
Since
\begin{align*}
\cO_X(L)\otimes I(X, B)&=\cO_X(K_X+B+A-\sum _{k=2}^{\infty}
m_k B_k)\\
&=\cO_X(K_X+B_1+\sum _{k=2}^{\infty}(k-m_k)B_k+A),
\end{align*}
In view of the Norimatsu vanishing theorem (cf.~\cite[Lemma 4.3.5]{lazarsfeld}), if we hope for vanishing, we should make $m_k$ equal $k$ or $k-1$ for every $k\geq 2$.
If $m_k=k$ for every $k\geq 2$, then
$$
I(X, B)=\mathcal J_{NLC}(X, B).
$$
If $m_k=k-1$ for every $k\geq 2$,
then
$$
I(X, B)=\mathcal J'(X, B).
$$
Let $f:Y\to X$ be a blow-up along a stratum of $\Supp B$,
where a {\em{stratum}} of $\Supp B$ means an lc center of $(X, \Supp B)$.
We put $K_Y+B_Y=f^*(K_X+B)$.
Then it is natural to require $$I(Y, B_Y)=\cO_Y(-\sum _{k=2}^{\infty}
n_k {}^k\!B_Y)$$ such that $n_k=k$ or $k-1$ for every $k\geq 2$ and
$$f_*I(Y, B_Y)=I(X, B). $$
We think that the most {\em{natural}} choices for non-lc ideal sheaves are
$$
I(X, B)=\mathcal J_{NLC}(X, B)=\cO_X(-\sum _{k=2}^{\infty}kB_k)
$$
or
$$
I(X, B)=\mathcal J'(X, B)=\cO_X(-\sum _{k=2}^{\infty}(k-1)B_k).
$$
The ideal sheaf $\mathcal J_{NLC}(X, B)$ should be called
{\em{minimal}} non-lc ideal sheaf of $(X, B)$ and
$\mathcal J'(X, B)$ should be called {\em{maximal}}
non-lc ideal sheaf of $(X, B)$.

The smaller $B_1+\sum _{k=2}^{\infty}(k-m_k)B_k$ is,
the more easily we can apply our torsion-free theorem
(cf.~Theorem \ref{ap1} (1)) to $I(X, B)$.
It is one of the main reasons why the first author adopted
$\mathcal J_{NLC}(X, B)$ to define
$\Nlc (X, B)$.

Finally, we put
\begin{align*}
I(X, B)=\cO_X(-\sum _{k=2}^{1-l}kB_k-\sum _{k=2-l}^{\infty}(k-1)B_k)
=:\mathcal J'_l(X, B)
\end{align*}
for $l=0, -1, \cdots, -\infty$.
Then
\begin{align*}
\mathcal J_{NLC}(X, B)=\mathcal J'_{-\infty}(X, B)
\subset \mathcal J'_l(X, B)\subset
\mathcal J'_0(X, B)=\mathcal J'(X, B)
\end{align*}
and $\mathcal J'_l(X, B)$ satisfies all the
above desired properties for every $l$.
We will discuss $\mathcal J'_l(X, B)$ for every
negative integer $l$ in Section \ref{sec3.5}.
We do not know whether $\mathcal J'_l(X, B)$ with
$l\ne 0, -\infty$ is useful or not for geometric
applications.
\end{say}

\section{Maximal non-lc ideal sheaves}\label{sec2}
Let us define {\em{maximal non-lc ideal sheaves}}.

\begin{defn}\label{def-a}
Let $X$ be a normal variety and
$\Delta$ an $\mathbb R$-divisor on $X$ such that
$K_X+\Delta$ is $\mathbb R$-Cartier.
Let $f:Y\to X$ be a resolution with $K_Y+\Delta_Y=f^*(K_X+\Delta)$ such
that $\Supp \Delta_Y$ is simple normal crossing.
Then we put
$$
\mathcal J'(X, \Delta)=f_*\cO_Y(\ulcorner
K_Y-f^*(K_X+\Delta)+\varepsilon F\urcorner)
$$
for $0<\varepsilon \ll 1$, where
$F=\Supp \Delta^{\geq 1}_Y$.
\end{defn}
It is easy to see that the right hand side does not depend on $\varepsilon$
if $0<\varepsilon \ll 1$.
We note that
\begin{align*}
\ulcorner K_Y-f^*(K_X+\Delta)+\varepsilon F\urcorner&=\ulcorner -\Delta_Y+\varepsilon F\urcorner\\
&=-\llcorner \Delta_Y\lrcorner+\sum _{k=1}^{\infty}{}^k\!\Delta_Y
\end{align*}
for $0<\varepsilon \ll 1$.
Therefore, we can
write
\begin{align*}
\mathcal J'(X, \Delta)&=f_*\cO_Y(\ulcorner -\Delta_Y+\varepsilon F\urcorner)\\
&=f_*\cO_Y(-\llcorner \Delta_Y\lrcorner+\sum _{k=1}^{\infty}
{}^k\!\Delta_Y).
\end{align*}
By Lemma \ref{lem-a} below,
$\mathcal J'(X, \Delta)$ does not depend
on the resolution
$f:Y\to X$.

We note that $$\mathcal J(X, \Delta)=f_*\cO_Y(-\llcorner \Delta_Y\lrcorner)$$ is the {\em{multiplier ideal sheaf}} associated to
the pair $(X, \Delta)$ and that
$$\mathcal J_{NLC}(X, \Delta)=f_*\cO_Y(-\llcorner \Delta_Y\lrcorner
+\Delta^{=1}_Y)$$ is the ({\em{minimal}}) {\em{non-lc ideal sheaf}} associated to
the pair $(X, \Delta)$ (cf.~Definition \ref{def61}). It is obvious that
$$
\mathcal J(X, \Delta)\subset
\mathcal J_{NLC}(X, \Delta)\subset \mathcal J'(X, \Delta)
$$
by the above definitions, and it is also easy to check that the definition of $\mJ'(X, \Delta)$ agrees with that given in \ref{maxnonlcdefsm} when $X$ is smooth and $\Delta$ is effective.

From now on, we assume that
$\Delta$ is effective.
Then $\mathcal J'(X, \Delta)$ is an ideal
sheaf on $X$.
We are mainly interested in the case
when $\Delta$ is effective due to following
fact.

\begin{lem}Assume that
$\Delta$ is effective. Then
$(X, \Delta)$ is log canonical if and only if
$\mathcal J'(X, \Delta)=\cO_X$.
\end{lem}

\begin{lem}\label{lem-a}
Let $X$ be a smooth variety and
$\Delta$ an $\mathbb R$-divisor
on $X$ such that
$\Supp \Delta$ is simple normal crossing.
Let $f:Y\to X$ be a proper birational
morphism such that
$\Exc (f)\cup \Supp f^{-1}_*\Delta$ is simple normal crossing.
We put $K_Y+\Delta_Y=f^*(K_X+\Delta)$.
Then
$$
f_*\cO_Y(\ulcorner -\Delta_Y+\varepsilon ' F'\urcorner)\simeq
\cO_X(\ulcorner -\Delta +\varepsilon F\urcorner)
$$
for $0<\varepsilon, \varepsilon'\ll 1$, where
$F=\Supp \Delta^{\geq 1}$ and $F'=\Supp \Delta^{\geq 1}_Y$.
\end{lem}
\begin{proof}
Since $K_Y+\Delta_Y=f^*(K_X+\Delta)$, we can
write
$$
K_Y=f^*(K_X+\{\Delta-\varepsilon F\}+\varepsilon F)
+f^*\llcorner \Delta-\varepsilon F\lrcorner-\Delta_Y.
$$
We note that $\{\Delta -\varepsilon F\}+\varepsilon F$ is a boundary
$\mathbb R$-divisor whose support is simple normal crossing for
$0<\varepsilon \ll 1$ and that
$(X, \{\Delta -\varepsilon F\})$ is klt.
Thus, $a(\nu, X, \{\Delta-\varepsilon F\}+\varepsilon F)\geq -1$ for
every $\nu$ (assuming again $0 < \varepsilon \ll 1$).
We can easily check that $a(\nu, X, F)=-1$
if $a(\nu, X, \{\Delta-\varepsilon F\}+\varepsilon F)=-1$ and
that $a(\nu, X, F)=-1$ induces $a(\nu, X, \Delta)\leq -1$
(cf.~\cite[Lemma 2.45]{km}).
Therefore, the round-up of
$f^*\llcorner \Delta-\varepsilon F\lrcorner-\Delta_Y+\varepsilon 'F'$ is
effective.
So, we can write
$$
f^*\llcorner \Delta-\varepsilon F\lrcorner-\llcorner
\Delta_Y-\varepsilon 'F'\lrcorner=E,
$$
where $E$ is an effective Cartier divisor on $Y$.
We can easily check that
$E$ is $f$-exceptional for $0<\varepsilon, \varepsilon'\ll1$.
Thus, we obtain
$$
f_*\cO_Y(\ulcorner -\Delta_Y+\varepsilon ' F'\urcorner)\simeq
\cO_X(\ulcorner -\Delta +\varepsilon F\urcorner)
$$
since $\ulcorner -\Delta_Y+\varepsilon 'F'\urcorner=f^*(\ulcorner
-\Delta+\varepsilon F\urcorner)+E$.
\end{proof}

We can also define $\mathcal J'$ for ideal sheaves.

\begin{defn}
Let $X$ be a normal variety and $\Delta$ an $\mathbb R$-divisor
on $X$ such that
$K_X+\Delta$ is $\mathbb R$-Cartier.
Let $\mathfrak a\subset \cO_X$ be a non-zero
ideal sheaf on $X$ and $c$ a real number.
Let $f:Y\to X$ be a resolution
such that $K_Y+\Delta_Y=f^*(K_X+\Delta)$ and
$f^{-1}\mathfrak a=\cO_Y(-E)$, where
$\Supp \Delta_Y \cup \Supp E$ has a simple normal crossing support.
We put
$$
\mathcal J'((X, \Delta); \mathfrak a^c)=f_*\cO_Y(-\llcorner \Delta_Y
+cE
\lrcorner+\sum _{k=1}^{\infty}{}^k\!(\Delta_Y+cE)).
$$
We sometime write
$$
\mathcal J'((X, \Delta); c\cdot \mathfrak a)=\mathcal J'((X, \Delta); \mathfrak a^{c}).
$$
\end{defn}

Of course, $\mathcal J'((X, \Delta); \mathfrak a^c)$ dose not
depend on $f:Y\to X$ by Lemma \ref{lem-a}.
We recall that
$$
\mathcal J((X, \Delta); \mathfrak a^c)=f_*\cO_Y
(-\llcorner \Delta_Y+cE\lrcorner).
$$

\begin{lem}\label{lem-ta}Let $X$ be a normal variety and
$\Delta$ an effective $\mathbb R$-divisor on $X$ such
that $K_X+\Delta$ is $\mathbb R$-Cartier.
Then we have
$$
\mathcal J'(X, \Delta)=\mathcal J((X, \Delta); \mathcal
J(X, \Delta)^{-\varepsilon})
$$
for
$0<\varepsilon \ll 1$.
In particular,
$$
\mathcal J'(X, \Delta)=\bigcap _{0<\varepsilon}\mathcal J((X, \Delta);
\mathcal J(X, \Delta)^{-\varepsilon}).
$$
\end{lem}

Although we will not use Lemma \ref{lem-ta} in the proof of
the restriction theorem (cf.~Theorem \ref{re-th}),
this lemma may help us understand $\mathcal J'$.

\begin{proof}
Let $f:Y\to X$ be a resolution with $f^{-1}\mathcal J(X, \Delta)=
\cO_Y(-E)$ such that
$\Exc (f)$, $\Supp f^{-1}_*\Delta$,
$\Supp E$,
and $\Exc (f)\cup \Supp f^{-1}_*\Delta\cup \Supp E$ are simple normal crossing
divisors.
We put $K_Y+\Delta_Y=f^*(K_X+\Delta)$. Then
$$
\mathcal J((X, \Delta); \mathcal J(X, \Delta)^{-\varepsilon})
=f_*\cO_Y(-\llcorner \Delta_Y-\varepsilon E\lrcorner)
$$
by definition.
Since $\Supp \Delta^{\geq 1}_Y\subset \Supp E$ and $\Delta^{<0}_Y$ is
$f$-exceptional, we can
easily check that
$$
f_*\cO_Y(-\llcorner \Delta_Y-\varepsilon E\lrcorner)
=f_*\cO_Y(\ulcorner -\Delta_Y+\varepsilon E\urcorner)=
\mathcal J'(X, \Delta).
$$
This completes the proof.
\end{proof}
By this lemma, $\mathcal J'(X, \Delta)$ itself is a
multiplier ideal sheaf.
The vanishing theorem holds for $\mathcal J'$.

\begin{thm}[Vanishing theorem]\label{vani2}
Let $X$ be a normal variety and $\Delta$
an $\mathbb R$-divisor on $X$
such that $K_X+\Delta$ is $\mathbb R$-Cartier.
Let $\pi:X\to S$ be a projective
morphism onto an algebraic variety $S$ and
$L$ a Cartier divisor on $X$.
Assume that $L-(K_X+\Delta)$ is $\pi$-ample.
Then we have
$$
R^i\pi_*(\mathcal J'(X, \Delta)\otimes \cO_X(L))=0
$$
for all $i>0$.
\end{thm}
\begin{proof}
Let $f:Y\to X$ be a resolution of $X$ with
$K_Y+\Delta_Y=f^*(K_X+\Delta)$ such
that $\Supp \Delta_Y$ is simple normal crossing.
Then
\begin{align*}
A-N+G+f^*L-(K_Y+\Delta^{=1}_Y
+\{\Delta_Y\}+G)=f^*(L-(K_X+\Delta)),
\end{align*}
where $A=\ulcorner -(\Delta^{<1}_Y)\urcorner$,
$N=\llcorner \Delta^{>1}_Y\lrcorner$,
and $G=\sum _{k=2}^{\infty}{^k}\Delta_Y$. Therefore,
$R^i\pi_*(f_*\cO_Y(A-N+G+f^*L))=0$ for
all $i>0$ by Theorem \ref{ap1} (2).
Thus, we obtain the desired vanishing theorem
since $$f_*\cO_Y(A-N+G+f^*L)
\simeq \mathcal J'(X, \Delta)\otimes \cO_X(L).$$
We finish the proof.
\end{proof}

\begin{rem}\label{777}
When $\Delta$ is effective in Theorem \ref{vani2},
the assumption that $L-(K_X+\Delta)$ is
$\pi$-ample
can be replaced by the following weaker assumption:~
$\pi$ is only proper, $L-(K_X+\Delta)$ is $\pi$-nef and $\pi$-big,
$(L-(K_X+\Delta))|_{\Nlc (X, \Delta)}$ is $\pi$-ample,
and $(L-(K_X+\Delta))|_C$ is $\pi$-big for every lc center
$C$ of $(X, \Delta)$. For details, see
\cite[Theorem 2.47]{fuji-book} and Remark \ref{rem24}.
\end{rem}

We close this section with the following simple example.
Here, we use the notation in \cite[9.3.C Monomial Ideals]{lazarsfeld}.

\begin{thm}\label{monomialchar0}
Let $\mathfrak a$ be a monomial ideal on $X=\mathbb C^n$.
Then $\mathcal J'(c\cdot \mathfrak a)=\mathcal J'((X, 0);
c\cdot \mathfrak a)$ is the monomial
ideal generated by all monomials $x^v$ whose exponent vectors satisfy
the condition that
$$
v+\mathbf {1}\in P(c\cdot \mathfrak a),
$$
where $P(c\cdot \mathfrak a)$ is the {\em{Newton polyhedron}} of
$c\cdot \mathfrak a$.
\end{thm}
\begin{proof}
It is obvious by Howald's theorem (cf.~\cite[Theorem 9.3.27]{lazarsfeld})
since $\mathcal J'(X, c\cdot \mathfrak a)=\mathcal J(X, (1-\varepsilon)
c\cdot \mathfrak a)$ for
$0<\varepsilon \ll 1$.
\end{proof}

\section{Intermediate non-lc ideal sheaves}\label{sec3.5}
This section is a continuation of Section \ref{sec1.5}.

\begin{say}\label{61}Let $X$ be a normal variety and $\Delta$
an $\mathbb R$-divisor on
$X$ such that $K_X+\Delta$ is $\mathbb R$-Cartier.
Let $f:Y\to X$ be a resolution with
$K_Y+\Delta_Y=f^*(K_X+\Delta)$ such that
$\Supp \Delta_Y$ is simple normal crossing.
Then we set
$$
\mathcal J'_l(X, \Delta)=f_*\cO_Y(-\llcorner
\Delta_Y\lrcorner +\Delta^{=1}_Y+\sum _{k=2-l}^{\infty}{}^k\!\Delta_Y)
$$ for
$l=0, -1, \cdots, -\infty$.
We have the natural inclusions
\begin{align*}
\mathcal J_{NLC}(X, \Delta)=&\mathcal J'_{-\infty}(X, \Delta)
\subset \cdots\subset \mathcal J'_l(X, \Delta)\\
&\subset \mathcal J'_{l+1}(X, \Delta)\subset \cdots \subset
\mathcal J'_0(X, \Delta)=\mathcal J'(X, \Delta).
\end{align*}
Of course, it is obvious that
there are only finitely many distinct ideals in the set $\{\mathcal J'_l(X, \Delta)\}_{l=0, -1, \cdots, -\infty}$.

We note that
$$
\ulcorner -\Delta_Y+\Delta ^{=1}_Y+
\varepsilon F\urcorner=
-\llcorner
\Delta_Y\lrcorner +\Delta^{=1}_Y+\sum _{k=2-l}^{\infty}{}^k\!\Delta_Y
$$
when $F=\Supp \Delta^{\geq 2-l}$ and
$0< \varepsilon \ll  1$.
Thus, $\mathcal J'_l(X, \Delta)$ is well-defined
by the following lemma.

\begin{lem}\label{lem-b}
Let $X$ be a smooth variety and
$\Delta$ an $\mathbb R$-divisor
on $X$ such that
$\Supp \Delta$ is simple normal crossing.
Let $f:Y\to X$ be a proper birational
morphism such that
$\Exc (f)\cup \Supp f^{-1}_*\Delta$ is simple normal crossing.
We put $K_Y+\Delta_Y=f^*(K_X+\Delta)$.
Then
$$
f_*\cO_Y(\ulcorner -\Delta_Y+\Delta ^{=1}_Y+
\varepsilon ' F'\urcorner)\simeq
\cO_X(\ulcorner -\Delta +\Delta^{=1}+\varepsilon F\urcorner)
$$
for $0<\varepsilon, \varepsilon'\ll 1$, where
$F=\Supp \Delta^{\geq m}$ and $F'=\Supp \Delta^{\geq m}_Y$ for
every positive integer $m\geq 2$.
\end{lem}
\begin{proof}
Since $K_Y+\Delta_Y=f^*(K_X+\Delta)$, we can
write
$$
K_Y=f^*(K_X+\{\Delta-\Delta^{=1}-\varepsilon F\}+\Delta^{=1}+\varepsilon F)
+f^*\llcorner \Delta-\Delta^{=1}-\varepsilon F\lrcorner-\Delta_Y.
$$
We note that $\{\Delta -\Delta^{=1}-\varepsilon F\}+\Delta^{=1}
+\varepsilon F$ is a boundary
$\mathbb R$-divisor whose support is simple normal crossing.
Thus, $$a(\nu, X,
\{\Delta -\Delta^{=1}-\varepsilon F\}+\Delta^{=1}
+\varepsilon F)\geq -1$$ for
every $\nu$.
We can easily check that $a(\nu, X, \Delta^{=1}+F)=-1$
if $a(\nu, X, \{\Delta -\Delta^{=1}-\varepsilon F\}+\Delta^{=1}
+\varepsilon F)=-1$ and
that $a(\nu, X, \Delta^{=1}+F)=-1$ induces $a(\nu, X, \Delta)= -1$
or $a(\nu, X, \Delta)\leq -m$ (cf.~\cite[Lemma 2.45]{km}).
Therefore, the round-up of
$f^*\llcorner \Delta-\Delta^{=1}-
\varepsilon F\lrcorner-\Delta_Y+\Delta^{=1}_Y+\varepsilon 'F'$ is
effective.
So, we can write
$$
f^*\llcorner \Delta-\Delta^{=1}-\varepsilon F\lrcorner-\llcorner
\Delta_Y-\Delta^{=1}_Y-\varepsilon 'F'\lrcorner=E,
$$
where $E$ is an effective Cartier divisor on $Y$.
We can easily check that
$E$ is $f$-exceptional for $0<\varepsilon, \varepsilon'\ll1$.
Thus, we obtain
$$
f_*\cO_Y(\ulcorner -\Delta_Y+\Delta^{=1}_Y
+\varepsilon ' F'\urcorner)\simeq
\cO_X(\ulcorner -\Delta +\Delta^{=1}+\varepsilon F\urcorner)
$$
since $\ulcorner -\Delta_Y+\Delta^{=1}_Y+\varepsilon 'F'\urcorner=f^*(\ulcorner
-\Delta+\Delta^{=1}+\varepsilon F\urcorner)+E$.
\end{proof}

The next property is obvious by the definition of
$\mathcal J'_l$.

\begin{lem}
Let $X$ be a normal variety and $\Delta$ an
effective $\mathbb R$-divisor on $X$ such
that $K_X+\Delta$ is $\mathbb R$-Cartier.
Then, for every $l$,
$(X, \Delta)$ is log canonical
if and only if $\mathcal J'_l(X, \Delta)=\cO_X$.
\end{lem}

We note the following Bertini type theorem.

\begin{lem}
Let $X$ be a normal variety and $\Delta$ an effective $\mathbb R$-divisor
on $X$ such that $K_X+\Delta$ is $\mathbb R$-Cartier.
Let $\Lambda$ be a linear system on $X$ and $D\in \Lambda$ a general
member of $\Lambda$. Then
$$\mathcal J'_l(X, \Delta)=\mathcal J'_l(X, \Delta +tD)$$ outside
the base locus $\Bs \Lambda$ of $\Lambda$ for
all $0\leq t\leq 1$ and $l$.
\end{lem}
\begin{proof}
By replacing $X$ with $X\setminus \Bs\Lambda$, we can assume
that $\Bs \Lambda=\emptyset$.
Let $f:Y\to X$ be a resolution as in \ref{61}.
Since $D$ is a general member of $\Lambda$,
$f^*D=f^{-1}_*D$ is a smooth divisor on $Y$ such that
$\Supp f^*D\cup \Supp \Delta_Y$ is simple normal
crossing.
Therefore, we can check that
$$
\mathcal J'_l(X, \Delta +tD)=\mathcal J'_l(X, \Delta)
$$
for all $0\leq t\leq 1$ and $l$.
\end{proof}

The vanishing theorem also holds for $\mathcal J'_l$.

\begin{thm}[Vanishing theorem]\label{vani3}
Let $X$ be a normal variety and $\Delta$
an $\mathbb R$-divisor on $X$
such that $K_X+\Delta$ is $\mathbb R$-Cartier.
Let $\pi:X\to S$ be a projective
morphism onto an algebraic variety $S$ and
$L$ a Cartier divisor on $X$.
Assume that $L-(K_X+\Delta)$ is $\pi$-ample.
Then we have
$$
R^i\pi_*(\mathcal J'_l(X, \Delta)\otimes \cO_X(L))=0
$$
for all $i>0$ and every $l$.
\end{thm}
\begin{proof}
We put $G=\sum _{k=2-l}^{\infty}{}^k\!\Delta_Y$. Then
the proof of Theorem \ref{vani2} works without any changes.
\end{proof}

Note that Remark \ref{777} works for every $l$.
We also note that the restriction theorem does not necessarily hold
for $l\ne 0, -\infty$.

\begin{ex}\label{intermediateexample}
Let $X=\mathbb C^2=\Spec \mathbb C[x, y]$,
$S=(x=0)$, $C=(y^2=x^3)$ and $B=2C$.
We put
$K_S+B_S=(K_X+S+B)|_S$,
and we will compare the intermediate non-lc ideals of $(X, S+B)$ and of $(S, B_S)$.

We consider the following sequence of blow-ups:
\[
X\stackrel{f_1}{\longleftarrow}X_1\stackrel{f_2}{\longleftarrow}X_2\stackrel{f_3}{\longleftarrow}X_3.
\]
We denote by $E_i$ the exceptional curve of $f_i$ (and we use the same letter for its strict transform).
Let $f_1: X_1 \to X$ be the blow-up at the origin,
$f_2:X_2 \to X_1$ be the blow-up at the intersection point of $E_1$ and $C$ and
$f_3: X_3 \to X_2$ be the blow-up at the intersection point of $E_1$, $E_2$ and $C$.
Then $\pi:=f_3 \circ f_2 \circ f_1: X_3 \to X$ is a log resolution of $(X, S+B)$, and we have
$K_{X_3/X}=E_1+2E_2+4E_3$, $\pi^*B=4E_1+6E_2+12E_3+B$ and $\pi^*S=E_1+E_2+2E_3+S$.
By the projection formula, we obtain
\begin{align*}
\mathcal{J}'_{-1}(X, S+B)&=\pi_*\mathcal{O}_{X_3}(-3E_1-4E_2-9E_3-B)\\
                                            &=\pi_*(\mathcal{O}_{X_3}(E_1+2E_2+3E_3) \otimes \pi^*\mathcal{O}_X(-B))\\
                                            &=\mathcal{O}_X(-B).
\end{align*}
On the other hand, since $B_S=(y^4=0)$ in $S$, one can easily see that
$$
\mathcal J'_{-1}(S, B_S)=\mathfrak m^3, \ \ \ \
$$
where $\mathfrak m$ is the maximal
ideal corresponding to $0\in S$.
Of course, we have
$$
\mathcal J'_{-1}(X, S+B)|_S=\cO_X(-B)|_S=\mathfrak m^4.
$$
Thus, we obtain
$$
\mathcal J'_{-1}(X, S+B)|_S\subsetneq \mathcal J'_{-1}(S, B_S).
$$
\end{ex}
\end{say}

\section{Supplementary remarks}\label{sec7}

This section is a supplement to \cite{fundamental}.
We consider various scheme structures on the non-lc locus of
the pair $(X, B)$.

\begin{defn}
Let $X$ be a normal variety and
$B$ an effective $\mathbb R$-divisor on $X$ such that
$K_X+B$ is $\mathbb R$-Cartier.
We put
$$
\cO_{\Nlc(X, B)_l}=\cO_X/\mathcal J'_l(X, B)
$$
for every $l=-\infty, \cdots, 0$.
We simply write
$$\cO_{\Nlc(X, B)}=\cO_{\Nlc(X, B)_{-\infty}}
=\cO_X/\mathcal J_{NLC}(X, B)
$$
and $$
\cO_{\Nlc(X, B)'}=\cO_{\Nlc(X, B)_0}=\cO_X/\mathcal J'_0(X, B)=\mathcal
O_X/\mathcal J'(X, B).
$$
We note that there exists natural surjection
$$
\cO_{\Nlc(X, B)_l}\to \cO_{\Nlc(X, B)_k}
$$
for every $l<k$.
\end{defn}

The following theorem is a slight generalization of
\cite[Theorem 8.1]{fundamental}.
The proof of \cite[Theorem 8.1]{fundamental}
works without any modifications.

\begin{thm}\label{722}
Let $l$ be an arbitrary non-positive integer or $-\infty$.
Let $X$ be a normal
variety and $B$ an effective $\mathbb R$-divisor
on $X$ such that
$K_X+B$ is $\mathbb R$-Cartier.
Let $D$ be a Cartier divisor on $X$.
Assume that $D-(K_X+B)$ is $\pi$-ample,
where $\pi:X\to S$ is a projective morphism onto a variety $S$.
Let $\{C_i\}$ be {\em{any}} set of
lc centers of the pair $(X, B)$.
We put $W=\bigcup C_i$ with the reduced scheme structure.
Assume that $W$ is disjoint from $\Nlc (X, B)$.

Then we have
$$
R^i\pi_*(\mathcal J\otimes \cO_X(D))=0
$$
for every $i>0$, where $\mathcal J=\mathcal I_W\cdot \mathcal
J'_l(X, B)\subset \cO_X$ and
$\mathcal I_W$ is the defining
ideal sheaf of $W$ on $X$.
Therefore, the restriction
map
$$
\pi_*\cO_X(D)\to \pi_*\cO_W(D)\oplus
\pi_*\cO_{\Nlc(X, B)_l}(D)
$$
is surjective and
$$
R^i\pi_*\cO_W(D)=0
$$
for every $i>0$.
In particular, the restriction maps
$$\pi_*\cO_X(D)\to \pi_*\cO_W(D)$$
and $$
\pi_*\cO_X(D)\to \pi_*\cO_{\Nlc(X, B)_l}(D)
$$
are surjective.
\end{thm}

We close this section with
the next supplementary result.
The proof is obvious.

\begin{prop}
In the non-vanishing theorem {\em{(}}cf.~\cite[Theorem 12.1]{fundamental}{\em{)}}
and the base point free theorem {\em{(}}cf.~\cite[Theorem 13.1]{fundamental}{\em{)}},
we assumed that
$\cO_{\Nlc(X, B)}(mL)$ is $\pi|_{\Nlc (X, B)}$-generated for every
$m\gg 0$.
However, it is sufficient to assume that
$\cO_{\Nlc (X, B)_l}(mL)$ is $\pi|_{\Nlc (X, B)_l}$-generated
for every $m\gg 0$,
where $l$ is any non-positive integer.
We note that, for every $l$,
$$
\pi_*\cO_X(mL)\to \pi_*\cO_{\Nlc(X, B)_l}(mL)
$$
is surjective for $m\geq a$ by the vanishing theorem
{\em{(}}cf.~{\em{Theorem \ref{722}}}{\em{)}}.
\end{prop}

Therefore, in \cite{fundamental},
we can adopt $\mathcal J'_l(X, B)$ for any $l$ instead of
$\mathcal J_{NLC}(X, B)$. However, from the point of view of the minimal model program,
we believe that
$\mathcal J_{NLC}(X, B)$ is the most natural defining ideal sheaf of the
non-lc locus $\Nlc (X, B)$ of $(X, B)$.
Also see Remark \ref{977} below.

\section{Non-klt ideal sheaves}\label{sec-klt}

In this section, we consider {\em{non-klt ideal sheaves}}.

\begin{say}\label{81}
Let $X$ be a normal variety and $\Delta$
an $\mathbb R$-divisor on
$X$ such that $K_X+\Delta$ is $\mathbb R$-Cartier.
Let $f:Y\to X$ be a resolution with
$K_Y+\Delta_Y=f^*(K_X+\Delta)$ such that
$\Supp \Delta_Y$ is simple normal crossing.
Then we put
$$
\mathcal J_l(X, \Delta)=f_*\cO_Y(-\llcorner
\Delta_Y\lrcorner+\sum _{k=2-l}^{\infty}{}^k\!\Delta_Y)
$$ for
$l=0, -1, \cdots, -\infty$.
We have the natural inclusions
\begin{align*}
\mathcal J(X, \Delta)=&\mathcal J_{-\infty}(X, \Delta)
\subset \cdots\subset \mathcal J_l(X, \Delta)\\
&\subset \mathcal J_{l+1}(X, \Delta)\subset \cdots \subset
\mathcal J_0(X, \Delta).
\end{align*}
We note that $\mathcal J(X, \Delta)=\mathcal J_{-\infty}(X, \Delta)$ is
the usual multiplier ideal sheaf associated to the pair $(X, \Delta)$.
Of course, it is obvious that
there are only finite numbers of
different ideals in $\{\mathcal J_l(X, \Delta)\}_{l=0, -1, \cdots, -\infty}$.
It is also obvious that
$$
\mathcal J_l (X, \Delta)\subset \mathcal J'_l(X, \Delta)
$$
for every $l$.
It is easy to check that $\mathcal J_l (X, \Delta)$ is well-defined, that is,
it does not depend on the resolution $f:Y\to X$ (cf.~Lemma \ref{lem-b}).
We note that
$\mathcal J_l(X,\Delta)\subset \cO_X$ when
$\Delta$ is effective.

\begin{lem}\label{lem-new}
Assume that $\Delta$ is effective. Then we have
$$
\mathcal J_{l} (X, \Delta)|_U=\mathcal J'_l (X, \Delta)|_U
$$ for
every $l$,
where $U=X\setminus W$ and $W$ is the union of all the lc centers of $(X, \Delta)$.
\end{lem}
\begin{proof}
By shrinking $X$, we can assume that $U=X$.
In this case, $f(\Delta^{=1}_Y)\subset \Nlc (X, \Delta)$.
Therefore, we see that $\mathcal J_l(X, \Delta)=\mathcal J'_l(X, \Delta)$
for every $l$.
\end{proof}
\begin{lem}
Let $X$ be a normal variety and $\Delta$ an effective $\mathbb R$-divisor
on $X$ such that $K_X+\Delta$ is $\mathbb R$-Cartier. Then,
for every $l$, $(X, \Delta)$ is kawamata log terminal if and only if
$\mathcal J_l (X, \Delta)=\cO_X$.
\end{lem}

We obtain the following vanishing theorem without any difficulties
as an application of Theorem \ref{ap1} (2).

\begin{thm}[Vanishing theorem]\label{833}
Let $X$ be a normal variety and $\Delta$
an $\mathbb R$-divisor on $X$
such that $K_X+\Delta$ is $\mathbb R$-Cartier.
Let $\pi:X\to S$ be a projective
morphism onto an algebraic variety $S$ and
$L$ a Cartier divisor on $X$.
Assume that $L-(K_X+\Delta)$ is $\pi$-ample.
Then we have
$$
R^i\pi_*(\mathcal J_l(X, \Delta)\otimes \cO_X(L))=0
$$
for all $i>0$ and every $l$.
\end{thm}
\begin{proof}
See the proof of Theorem \ref{vani2}.
\end{proof}

\begin{rem}
When $\Delta$ is effective in Theorem \ref{833},
the assumption that $L-(K_X+\Delta)$ is $\pi$-ample
can be replaced by the following weaker assumption:~$\pi$ is
only proper,
$L-(K_X+\Delta)$ is $\pi$-nef and $\pi$-big, and
$(L-(K_X+\Delta))|_{\Nlc (X, \Delta)}$ is $\pi$-ample
(cf.~\cite[Theorem 2.47]{fuji-book} and Remark \ref{rem24}).
It is well known that it is sufficient to assume
$L-(K_X+\Delta)$ is $\pi$-nef and $\pi$-big for
$l=-\infty$.
It is nothing but the Kawamata--Viehweg--Nadel vanishing theorem.
\end{rem}
\end{say}

\begin{say}We close this section with the following very important remark.
\begin{rem}\label{977}
Let $D$ be an effective
$\mathbb Q$-divisor
on a smooth variety $X$.
Definition 9.3.9 in \cite{lazarsfeld} says that
the pair $(X, D)$ is log canonical if $\mathcal
J(X, (1-\varepsilon)D)=\cO_X$ for all $0<\varepsilon <1$.

Again note that $\mathcal J'(X, D)|_U$ does not always coincide with
$\mathcal J(X, D)|_U$, where
$U=X\setminus W$ and
$W$ is the union of all the lc centers of $(X, D)$.
This may be a desirable property in certain circumstances, and so
the scheme structure on $\Nlc (X, D)$ induced by
$\mathcal J'(X, D)=\mathcal J(X, (1-\varepsilon)D)$ for
$0<\varepsilon \ll 1$ may be less suitable in some applications than the scheme structure induced by $\mathcal J_{NLC}(X, D)$.
\end{rem}
\end{say}

\section{Differents}\label{sec-diff}
Let us recall the definition and the basic
properties of Shokurov's {\em{differents}}
following \cite[\S 3]{shokurov} and
\cite[9.2.1]{ambro}. See also \cite[Section 14]{fundamental}.

\begin{say}[Differents]
Let $X$ be a normal variety and
$S+B$ an $\mathbb R$-divisor on $X$ such that
$K_X+S+B$ is $\mathbb R$-Cartier.
Assume that
$S$ is reduced and that
$S$ and $B$ have no common irreducible
components.
Let $f:Y\to X$ be a resolution
such that
$$
K_Y+S_Y+B_Y=f^*(K_X+S+B)
$$
and that
$\mathrm{Supp}(S_Y+B_Y)$ is simple normal crossing and $S_Y$ is
smooth, where
$S_Y$ is the strict transform of $S$ on $Y$.
Let $\nu:S^{\nu}\to S$ be the normalization.
Then $f:S_Y\to S$ can be decomposed as
$$
f:S_Y\overset{\pi}\longrightarrow S^{\nu}\overset{\nu}\longrightarrow S.
$$
We define $B_{S_Y}=B_Y|_{S_Y}$. Then
we obtain
$$
(K_Y+S_Y+B_Y)|_{S_Y}=K_{S_Y}+B_{S_Y}
$$
by adjunction. We put $B_{S^{\nu}}=\pi_*B_{S_Y}$.
Then we obtain that
$$
K_{S^{\nu}}+B_{S^{\nu}}=\nu^*(K_X+S+B).
$$
The $\mathbb R$-divisor $B_{S^{\nu}}$ on $S^{\nu}$ is called
the {\em{different}} of $(X, S+B)$ on $S^{\nu}$. We can easily
check that $B_{S^\nu}$ is independent of the resolution
$f:Y\to X$. So,
$B_{S^\nu}$ is a well-defined $\mathbb R$-divisor on $S^{\nu}$.
We can check the following properties.
\begin{itemize}
\item[(i)] $K_{S^\nu}+B_{S^\nu}$ is
$\mathbb R$-Cartier and $K_{S^{\nu}}+B_{S^\nu}=\nu^*(K_X+S+B)$.
\item[(ii)] If $B$ is a $\mathbb Q$-divisor, then
so is $B_{S^{\nu}}$.
\item[(iii)] $B_{S^\nu}$ is effective if $B$ is
effective in a neighborhood of $S$.
\item[(iv)] $(S^\nu, B_{S^\nu})$ is log canonical
if $(X, S+B)$ is log canonical
in a neighborhood of $S$.
\item[(v)] Let $D$ be an $\mathbb R$-Cartier divisor
on $X$ such that
$S$ and $D$ have no common irreducible components. Then
we have
$$
(B+D)_{S^\nu}=B_{S^\nu}+\nu^*D.
$$
We sometimes write $D|_{S^{\nu}}=\nu^*D$ for simplicity.
\end{itemize}
The properties except (iii) follow directly from the definition.
We give a proof of (iii) for the reader's convenience.

\begin{proof}[Proof of {\em{(iii)}}]
By shrinking $X$, we can assume that
$X$ is quasi-projective and
$B$ is effective.
By taking hyperplane cuts, we can also assume
that
$X$ is a surface.
Run the log minimal model program
over $X$ with respect to
$K_Y+S_Y$. Let $C$ be a curve
on $Y$ such that
$(K_Y+S_Y)\cdot C<0$ and $f(C)$ is a
point.
Then $K_Y\cdot C<0$ because $S_Y$ is the strict transform
of $S$. Therefore,
each step of the log minimal model
program over
$X$ with
respect to $K_Y+S_Y$ is the contraction of a $(-1)$-curve $E$
with $(K_Y+S_Y)\cdot E<0$.
So, by replacing $(Y, S_Y)$ with
the output of
the above log minimal model program,
we can assume that $Y$ is smooth,
$(Y, S_Y)$ is plt,
and $K_Y+S_Y$ is $f$-nef. We
note that
$S_Y$ is a smooth
curve since $(Y, S_Y)$ is plt.
By the negativity lemma and
the assumption
that $B$ is effective,
$B_Y$ is effective.
We note the following equality
$$
-B_Y=K_Y+S_Y-f^*(K_X+S+B).
$$
By adjunction, we obtain
$$
(K_Y+S_Y+B_Y)|_{S_Y}=K_{S_Y}+B_Y|_{S_Y}.
$$
It is obvious that $B_Y|_{S_Y}$ is effective.
This implies that $B_{S^{\nu}}=B_Y|_{S_Y}$ is effective.
\end{proof}
When $X$ is singular, $B_{S^\nu}$ is not necessarily zero
even if $B=0$.
\end{say}
\begin{say}[Inversion of adjunction]
Let us recall Kawakita's inversion of adjunction on log canonicity
(see \cite{kawakita}).

\begin{thm}\label{kawa-thm}
Let $X$ be a normal variety, $S$ a reduced divisor on $X$, and $B$
an effective $\mathbb R$-divisor
on $X$ such that
$K_X+S+B$ is $\mathbb R$-Cartier.
Assume that $S$ has no common irreducible component with the
support of $B$.
Let $\nu: S^{\nu}\to S$ be the normalization and
$B_{S^{\nu}}$ the different on $S^{\nu}$ such that
$K_{S^{\nu}}+B_{S^{\nu}}=\nu^*(K_X+S+B)$.
Then $(X, S+B)$ is log canonical
in a neighborhood of $S$ if and only if
$(S^{\nu}, B_{S^{\nu}})$ is log canonical.
\end{thm}

By adjunction, it is obvious that $(S^{\nu}, B_{S^{\nu}})$ is log canonical
if $(X, S+B)$ is log canonical
in a neighborhood of
$S$. It is the property (iv) above.
So, the above theorem is usually called the
{\em{inversion of adjunction on log canonicity}}.
We used Theorem \ref{kawa-thm} in the proof of the restriction theorem
for $\mathcal J_{NLC}$:~Theorem \ref{main} (see \cite{fujino-non}).
\end{say}

\section{Restriction theorems}\label{sec3}
In this section, we consider the restriction theorem for $\mathcal J'$.
First, let us recall the restriction theorem for $\mathcal J_{NLC}$.
It is the main theorem of \cite{fujino-non}.

\begin{thm}[Restriction theorem]\label{main}
Let $X$ be a normal variety and $S+B$ an effective $\mathbb R$-divisor
on $X$ such that $S$ is reduced and normal and that $S$ and $B$ have no
common irreducible components.
Assume that $K_X+S+B$ is $\mathbb R$-Cartier.
Let $B_S$ be the different on $S$ such that
$K_S+B_S=(K_X+S+B)|_S$.
Then we obtain
$$
\mathcal J_{NLC}(S, B_S)= \mathcal J_{NLC}(X, S+B)|_S.
$$
In particular, $(S, B_S)$ is log canonical if and only if
$(X, S+B)$ is log canonical in a neighborhood of $S$.
\end{thm}

There is a natural question on $\mathcal J'$.

\begin{que}\label{112}
Let $X$ be a normal variety and $S+B$ an effective $\mathbb R$-divisor
on $X$ such that $S$ is reduced and normal and that $S$ and $B$ have no
common irreducible components.
Assume that $K_X+S+B$ is $\mathbb R$-Cartier.
Let $B_S$ be the different on $S$ such that
$K_S+B_S=(K_X+S+B)|_S$.
Is the following equality
$$
\mathcal J'(S, B_S)= \mathcal J'(X, S+B)|_S.
$$
true?
\end{que}

In this section, we will give partial answers to Question \ref{112}.

\begin{say}
We prove the restriction theorem for $\mathcal J'$ under the assumption that $X$ is smooth and
$B=0$. The following theorem is contained in the main theorem from the next section, but the proof in this special case is sufficiently simple that we reproduce it here.

\begin{thm}
Let $X$ be a smooth variety and $S$ a reduced normal divisor
on $X$.
Then we have
$$
0\to \cO_X(-S)\to \mathcal J'(X, S)\to
\mathcal J'(S, 0)\to 0.
$$
In particular, we obtain
$$
\mathcal J'(X, S)|_S=\mathcal J'(S, 0)
$$
We note that $K_S=(K_X+S)|_S$ by adjunction.
\end{thm}
\begin{proof}
Let $f:Y\to X$ be a resolution of $S$ such that
$f$ is an isomorphism outside the
singular locus of $S$.
We put $E=\Exc (f)$.
We can assume that
$f$ is a composition
of blow-ups.
Each step can be assumed to be the blow-up described in
Proposition \ref{vani-lemlem} below.
We note that the codimension
of the singular locus of $S$ in $X$ is $\geq 3$ since
$S$ is normal.
Therefore, we have
$R^1f_*\cO_Y(K_Y+E)=0$ by
Proposition \ref{vani-lemlem}
and the Leray spectral sequence.
We consider the following
short exact sequence
$$
0\to \cO_Y(K_Y+E)\to \cO_Y(K_Y+S_Y+E)\to
\cO_{S_Y}(K_{S_Y}+E|_{S_Y})\to 0,
$$
where $S_Y=f^{-1}_*S$. By
taking $\otimes \cO_Y(-f^*(K_X+S))$ and
applying $f_*$,
we obtain
$$
0\to \cO_X(-S)\to \mathcal J'(X, S)\to
\mathcal J'(S, 0)\to 0.
$$
In particular, we have
$$
\mathcal J'(X, S)|_S=\mathcal J'(S, 0).
$$
We finish the proof.
\end{proof}

\begin{prop}[Vanishing lemma]\label{vani-lemlem}
Let $V$ be a smooth variety and $D$ a simple normal crossing
divisor on $V$. Let $f:W\to V$ be the blow-up along
$C$ such that
$C$ is smooth, irreducible, and
has simple normal crossings with $D$. We put
$F=f^{-1}_*D+E$, where
$E$ is the exceptional divisor of $f$.
We further assume that, if $C \not\subset D$, then the codimension of $C$ is $\geq 3$.
We then obtain
$$
f_*\cO_W(K_W+F)\simeq \cO_V(K_V+D)
$$
and
$$
R^1f_*\cO_W(K_W+F)=0.
$$
Note that $W$ is smooth and $F$ is a simple normal crossing divisor
on $W$.
\end{prop}
We provide two proofs, the first relies on more standard methods while the second relies on the theory of Du~Bois singularities.
\begin{proof}[Proof \#1]
First, we can easily check that
$$
K_W+F=f^*(K_V+D)+G,
$$
where $G$ is an effective $f$-exceptional
Cartier divisor.
Therefore, we obtain
$$
f_*\cO_W(K_W+F)\simeq \cO_V(K_V+D).
$$

Next, we consider the following
short exact sequence
$$
0\to \cO_W(K_W+f^{-1}_*D)\to \cO_W(K_W+F)\to
\cO_E(K_E+f^{-1}_*D|_E)\to 0.
$$
It is sufficient to prove
$$
R^1f_*\cO_W(K_W+f^{-1}_*D)=R^1f_*\cO_E(K_E+f^{-1}_*D|_E)=0.
$$
Assume that
$C\not\subset D$. Then
$f^{-1}_*D=f^*D$. In this case, $R^1f_*\cO_W(K_W+f^*D)=0$ by the
Grauert--Riemenschneider vanishing theorem and
$R^1f_*\cO_E(K_E+f^*D|_E)=0$ since
$f:E\to C$ is a $\mathbb P^m$-bundle
with $m\geq 2$.
Here, we used the assumption that
the codimension of $C$ in $V$ is $\geq 3$.
So, we can assume that $C\subset D$.
In this case, $R^1f_*\cO_W(K_W+f^{-1}_*D)=0$ by the
vanishing theorem of Reid--Fukuda type (cf.~\cite[Lemma]{fukuda}). 
On the other hand, $R^1f_*\cO_E(K_E+f^{-1}_*D|_E)=0$ by
the relative Kodaira vanishing theorem. We note
that $\pi=f|_E:E\to C$ is a $\mathbb P^m$-bundle
for some $m\geq 1$ and $\pi(f^{-1}_*D)=C$. Thus,
$f^{-1}_*D|_E$ is $\pi$-ample.

We have now proved that $R^1f_*\cO_W(K_W+F)=0$.
\end{proof}

\begin{proof}[Proof \#2]
 Note that the map $f$ is a log resolution of the scheme $D \cup C$.
 The scheme $D \cup C$ is in simple normal crossings so it
 has Du~Bois singularities.
 It follows that $R f_* \cO_F \qis \cO_{D \cup C}$ by \cite{SchwedeEasyCharacterization}.
 Thus Grothendieck duality implies that
\begin{equation}
\label{eqnDuBoisImpliesIso}
R f_* \omega_F^{\mydot} \qis \omega_{D \cup C}^{\mydot}.
\end{equation}
We can map the isomorphism from Equation (\ref{eqnDuBoisImpliesIso})
into the isomorphism $R f_* \omega_W^{\mydot} \qis \omega_V^{\mydot}$,
and then, from the resulting exact triangle,
obtain $R f_* \cO_W(K_W + F)[\dim W] \qis R f_* \omega_W^{\mydot}(F)
\qis R \sHom_W^{\mydot}(\I_{D \cup C}, \omega_V^{\mydot})$.

We assume that $C \not\subseteq D$
(as the other case is even easier).
The dualizing complex of $D \cup C$ has zero cohomology for $i$
between the degrees $-\dim D$ and $-\dim C$.  To see this, simply take cohomology and form the long exact sequence from the triangle
\[
\xymatrix{
\omega_{D \cap C}^{\mydot} \ar[r] & \omega_D^{\mydot} \oplus \omega_C^{\mydot} \ar[r] & \omega_{D \cup C}^{\mydot} \ar[r]^-{+1} &
}
\]
noting that $\myH^i( \omega_Z^{\mydot}) = 0$ for all $i < -\dim Z$.
Thus
$$0 \cong R^{-\dim D + 1} f_* \omega_F^{\mydot}
\cong R^{-\dim D} f_* \cO_W(K_W + F)[\dim W] = R^{1} f_* \cO_W(K_W + F)$$
and we have proven the vanishing.
For the isomorphism, simply notice that
\begin{align*}
 f_* \cO_W(K_W + F) &\cong \myH^{-\dim V}
 \left(R \sHom_V^{\mydot}(\I_{D \cup C}, \omega_V^{\mydot})\right)\\
 &\cong \sHom_W(\I_{D \cup C}, \omega_V)\\
 &\cong \cO_V(K_V + D)
\end{align*}
since the last two sheaves are reflexive and agree outside of $C$.
\end{proof}

\end{say}

\begin{say}
We prove the restriction theorem for $\mathcal J'$
on the assumption that $X$ has only mild singularities.
Theorem \ref{re-th} is an easy corollary of Theorem \ref{main} and is not
covered by the main result of the next section.

\begin{thm}\label{re-th}
Let $X$ be a normal variety and $S+B$ an effective
$\mathbb R$-divisor on $X$ such that
$S$ is reduced and normal
and that $S$ and $B$ have no common irreducible components.
Assume that
$B=B_1+B_2$ such that both
$B_1$ and $B_2$ are effective
$\mathbb R$-divisors around $S$,
$(X, S+B_1)$ is log canonical in a neighborhood of
$S$.
Then we obtain
$$
\mathcal J'(X, S+B)|_S=\mathcal J'(S, B_S),
$$
where $B_S$ is the different on $S$ such that
$(K_X+S+B)|_S=K_S+B_S$.
\end{thm}

\begin{lem}\label{oi}
With the same notation and assumptions
as in {\em{Theorem \ref{re-th}}},
we have
$$
\mathcal J'(X, S+B)=\mathcal J_{NLC}(X, S+B_1+(1-\varepsilon)B_2)
$$
in a neighborhood of $S$ for $0<\varepsilon \ll 1$, and
$$
\mathcal J'(S, B_S)=
\mathcal J'(S, {B_1}_S+{B_2}|_S)=\mathcal J_{NLC}(S, {B_1}_S+
(1-\varepsilon)B_2|_S)
$$
for $0<\varepsilon \ll 1$.
\end{lem}

\begin{proof}
By shrinking $X$ around $S$, we can assume that
$(X, S+B_1)$ is log canonical and $B_2$ is effective.
By the definitions of $\mathcal J'$ and $\mathcal J_{NLC}$, it is almost
obvious that
$$
\mathcal J'(X, S+B)=\mathcal J_{NLC}(X, S+B_1+(1-\varepsilon)B_2)
$$
for $0<\varepsilon \ll 1$.
By the assumption, $(S, B_{1S})$ is log canonical, where
$B_{1S}$ is the different such that $(K_X+S+B_1)|_S=K_S+B_{1S}$.
Thus, $\mathcal J'(S, B_S)=\mathcal J_{NLC}(S, B_{1S}+(1-\varepsilon)B_2|_S)$ holds for
$0<\varepsilon \ll 1$.
\end{proof}

\begin{proof}[Proof of {\em{Theorem \ref{re-th}}}]
We have the following equalities.
\begin{align*}
\mathcal J'(X, S+B)|_S&=\mathcal J_{NLC}(X, S+B_1+
(1-\varepsilon)B_2)|_S\\
&=\mathcal J_{NLC}(S, {B_1}_S+(1-\varepsilon)B_2|_S)
\\  &=\mathcal J'(S, {B_1}_S+B_2|_S) \\
&=\mathcal J'(S, B_S)
\end{align*}
by Lemma \ref{oi} and Theorem \ref{main}.
\end{proof}
\end{say}

We close this section with the following nontrivial example.

\begin{ex}\label{restrictionex}
Let $X=\mathbb C^2=\Spec \mathbb C[x, y]$,
$S=(x=0)$,
and $C=(y^2=x^3)$. We put
$K_S+C_S=(K_X+S+C)|_S$.
We use the same notation as in Example \ref{intermediateexample}.
Then we have
\begin{align*}
\mathcal{J}_{NLC}(X, S+C)&=\pi_*\mathcal{O}_{X_3}(-2E_1-2E_2-4E_3)\\
                                               &={f_1}_*\mathcal{O}_{X_1}(-2E_1)\\
                                               &=\mathfrak{n}^2, \\
\mathcal{J}'(X, S+C)&=\pi_* \mathcal{O}_{X_3}(-E_1-E_2-3E_3)\\
                                    &=\pi_* \mathcal{O}_{X_3}(-E_1-2E_2-3E_3)\\
                                    &=f_{1*} (f_{2*} \mathcal{O}_{X_2}(-E_2)
                                    \otimes \mathcal{O}_{X_1}(-E_1))\\
                                    &=(x^2, y),
\end{align*}
where $\mathfrak{n}$ is the maximal ideal corresponding to $(0,0) \in X$.
On the other hand, by easy calculations, we obtain
$$
\mathcal J_{NLC}(S, C_S)=\mathfrak m^2, \ \ \ \
\mathcal J'(S, C_S)=\mathfrak m,
$$
where $\mathfrak m$ is the maximal ideal corresponding to $0\in S$.
Hence we can check the both restriction theorems (cf.~Theorem \ref{re-th})
\begin{align*}
\mathcal J_{NLC}(S, C_S)&=\mathcal J_{NLC}(X, S+C)|_S,\\
\mathcal J'(S, C_S)&=\mathcal J'(X, S+C)|_S
\end{align*}
in this case.
\end{ex}

\section{The restriction theorem for complete intersections} \label{KarlSection}

In this section we prove a restriction theorem,
Theorem \ref{thmRestrictionForMaximalNonLC},
for maximal non-lc ideals $\mJ'(X, D)$ in a complete intersection.
It is important to note that we do not use Kawakita's proof
of inversion of adjunction for log canonicity \cite{kawakita}.
We also only use fairly mundane vanishing theorems -- Kawamata--Viehweg
vanishing in the form of local vanishing for multiplier
ideals, see \cite{KawamataVanishing},
\cite{ViehwegVanishingTheorems}, and \cite{lazarsfeld}.

Our method is related to techniques used to study Du~Bois
singularities, and so some of the auxiliary notation we use
draws from this perspective, for an introduction to Du~Bois singularities,
see \cite{KovacsSchwedeDuBoisSurvey}.
We briefly recall why one might expect to use techniques
from Du~Bois singularities to study non-lc ideal sheaves.

Suppose that $X$ is a reduced scheme of finite type
over a field $k$ of characteristic zero.  One can then associate an object
$\DuBois{X}$ in the bounded derived category with coherent
cohomology (this object has its origin in Deligne's mixed Hodge
theory for singular varieties).  This object $\DuBois{X}$ is used to
determine whether or not $X$ has Du~Bois singularities
(recalling that Du~Bois singularities are very closely
related to log canonical singularities, see \cite{KollarKovacsLCImpliesDB}).
Furthermore, in the case that $X$ is
normal and $K_X$ is Cartier, it follows
from \cite{KovacsSchwedeSmithLCImpliesDuBois} that the most
interesting cohomology ($-\dim X$) of the Grothendieck
dual of $\DuBois{X}$ is equal to $\mJ'(X, 0)$.  This suggests that two things:
\begin{itemize}
\item{} $\mJ'(X, 0)$ is natural object from the point of view of Du~Bois singularities (or more generally, from the point of view of the Hodge theory of singular varieties), and
\item{} some of the ideas from Du Bois singularities might be
useful in proving restriction theorems for $\mJ'(X, \Delta)$.
\end{itemize}
We will take advantage of the second idea in this section.

Now we begin our main definitions.

Suppose that $Y$ is a smooth affine variety and $X \subset Y$ is a
reduced closed subscheme with ideal sheaf $\I_X$.
Let $\ba$ be an ideal on $Y$ and $t > 0$ a real number.
Let $\pi : \tld Y \rightarrow Y$ be a log resolution of
$(Y, X, \ba^t)$ and set $\ba \cO_{\tld Y} = \cO_{\tld Y}(-G)$
and $\I_X \cO_{\tld Y} = \cO_{\tld Y}(-\overline X)$.

Consider the following short exact sequence:
\[
0 \rightarrow \cO_{\tld Y}(\lfloor tG - \epsilon \overline X\rfloor)
\rightarrow \cO_{\tld Y}(\lfloor tG \rfloor) \rightarrow M_{X, \ba^t} \rightarrow 0,
\]
where $\epsilon > 0$ is arbitrarily small.  Furthermore, one can replace
$\overline X$ with the reduced pre-image of $X$, and assuming $\epsilon$
was chosen to be sufficiently small, the sequence does not change.

\begin{defn}
We define $\DuBois{X, \ba^t}$ to be $\myR \pi_* M_{X, \ba^t}$.
\end{defn}

\begin{rem}
 The object $\DuBois{X, \ba^t}$ is purely an auxiliary object from the point of
 view of this paper.  However, in the case that $\ba = \cO_Y$,
 it agrees with $\DuBois{X}$, the zeroth graded piece of the Deligne--Du Bois
 complex, see \cite{DuBoisMain}, \cite{EsnaultHodgeTypeOfSubvarieties90}
 and \cite{SchwedeEasyCharacterization}.
\end{rem}

\begin{lem}
The object $\DuBois{X, \ba^t}$ is independent of the
choice of $\pi$ {\em{(}}assuming $\epsilon$
is chosen sufficiently small{\em{)}}.
\end{lem}
\begin{proof}
Suppose that $\rho : Y' \rightarrow \tld Y$ is a further log resolution.
Set $\ba \cO_{Y'} = \cO_{Y'}(-G')$ and $\I_X \cO_{Y'} = \cO_{Y'}(-X')$.
It is sufficient to show that
$\myR \rho_* \cO_{Y'}(\lfloor tG' - \epsilon X' \rfloor)
\qis \cO_{\tld Y}(\lfloor tG - \epsilon \overline X\rfloor)$.
Therefore, by Grothendieck duality,
it is sufficient to show that
\[
\myR \rho_* \cO_{Y'}(\lceil K_{Y'} - tG' + \epsilon X' \rceil)
\rightarrow \cO_{\tld Y}(\lceil K_{\tld Y} - tG + \epsilon \overline X\rceil)
\]
is an isomorphism.
Twisting by $\cO_{\tld Y}(-K_{\tld Y} - \overline X)$, it is sufficient to show that
\[
\myR \rho_* \cO_{Y'}(\lceil K_{Y'/\tld Y} - tG' - (1-\epsilon) X' \rceil)
\rightarrow \cO_{\tld Y}(\lceil - tG - (1 -\epsilon) \overline X\rceil)
\]
is an isomorphism.  But this is just the independence of the
choice of resolution for multiplier ideals.
\end{proof}

Instead of computing a full log resolution, it will be convenient
only to compute a log resolution of $(Y, \ba^t)$ which is an embedded resolution
of $X$.  In our next lemma, we show that our auxiliary object $\DuBois{X, \ba^t}$
can be computed via such a resolution.  But first a definition:

\begin{defn}
\label{defnFactorizingResolution}
 Given a closed subvariety $X$ in a smooth variety $Y$, a
 \emph{factorizing embedded resolution} of $X$ in $Y$ 
 is a proper birational morphism $\pi: \tld Y \to Y$  such that 
 \begin{enumerate}
\item $\tld Y$ is smooth and $\pi$ is an isomorphism at every generic point of $X \subseteq Y$, 
\item the exceptional locus $\mathrm{exc}(\pi)$ is a simple normal crossings divisor, 
\item the strict transform $\tld X$ of $X$ in $\tld Y$ is smooth and has simple normal crossings with $\mathrm{exc}(\pi)$, 
\item $\I_X \cO_{\tld Y} = \I_{\tld X} \cO_{\tld Y}(-E)$ where $\I_X$  (resp. $\I_{\tld X}$) is the ideal sheaf of $X$ (resp. $\tld X$) and $E$ is a $\pi$-exceptional divisor on $\tld Y$.  
\end{enumerate}

Such resolutions always exist,
see \cite{BravoEncinasVillmayorSimplified} or \cite{WlodarczykResolution}.
\end{defn}

\begin{lem}
\label{LemmaEmbeddedResolutionIsOk}
Suppose that $X$, $Y$ and $\ba$ are as above
and suppose that no component of $X$ is contained in $V(\ba)$.
Further suppose that $\pi : \tld Y \rightarrow Y$ is a
log resolution of $(Y, \ba^t)$ that is simultaneously a factorizing
embedded resolution of $X$ in $Y$ {\em{(}}and so that the pullbacks of
the various objects we were working with are in simple normal crossings,
see \cite{BravoEncinasVillmayorSimplified} or \cite{WlodarczykResolution}{\em{)}}.
Set $\tld X$ to be the strict transform of $X$ in
$\tld Y$ so that we can write $\I_X \cO_{\tld Y} = \I_{\tld X} \cO_{\tld Y}(-E)$.
Then if we let $M$ be the cokernel of
\[
\cO_{\tld Y}(\lfloor tG - \epsilon E \rfloor) \tensor \I_{\tld X} \rightarrow \cO_{\tld Y}(\lfloor tG \rfloor)
\]
we have that $\myR \pi_* M \qis \DuBois{X, \ba^t}$.
{\em{(}}Note the statement does not change if we replace $E$ by $E_{\red}$.{\em{)}}
\end{lem}
\begin{proof}
We note that by blowing up $\tld X$ (which is smooth),
we can obtain an actual log resolution $\eta : Y' \rightarrow Y$ of $(Y, X, \ba^t)$
as pictured in the diagram below:
\[
\xymatrix{
Y' \ar[rr]^{\rho} \ar[dr]_{\eta} & & \tld Y \ar[dl]^{\pi} \\
& Y & \\
}
\]
Set $\I_{\tld X} \cO_{Y'} = \cO_{Y'}(-X')$, set
$\cO_{Y'}(-E') = \rho^* \cO_{\tld Y}(-E)$ and set
$\ba\cO_{Y'} = \cO_{Y'}(-G') = \rho^* \cO_{\tld Y}(-G)$.
Note that $\rho$ induces a bijection between the components (and coefficients)
of $E$ with those of $E'$ (and also of $G$ with those of $G'$ since no
component of $X$ is contained in $V(\ba)$).
It is sufficient to show that
\[
\myR \rho_* \cO_{Y'}(\lfloor tG' - \epsilon(E' + X') \rfloor)
\qis \cO_{\tld Y}(\lfloor tG - \epsilon E \rfloor) \tensor \I_{\tld X}.
\]
Now twist by $\cO_{\tld Y}(-\lfloor tG - \epsilon E \rfloor)$, it is thus sufficient to show that
\[
\myR \rho_* \cO_{Y'}(\lfloor tG' - \epsilon(E' + X')
\rfloor - \rho^* \lfloor tG - \epsilon E \rfloor)) \qis \I_{\tld X}.
\]
Note that, over each component of $\tld X$, there is exactly one new divisor
created by $\rho$, they are all disjoint, and $\lfloor tG - \epsilon E \rfloor$
does not contain $\I_{\tld X}$
in its support, thus the left side is just $\myR \rho_* \cO_{Y'}(\lfloor -\epsilon X' \rfloor)$
which is isomorphic to $\I_{\tld X}$ since $\tld X$ is smooth.
\end{proof}

In what follows, we use the symbol $\myD$ to denote the
Grothendieck dual of a complex in $D^{b}_{\coherent}$,
for example $\myD(\DuBois{X}) \cong \myR \sHom_Y^{\mydot}(\DuBois{X}, \omega_Y^{\mydot})$.
See \cite{HartshorneResidues}.

We now make a transition in concept.  Instead of simply
looking at $\DuBois{X, \ba^{t}}$, we consider
$\DuBois{X, \ba^{t - \epsilon'}}$.  In particular, viewing
$\epsilon'$ as a very small positive number.  It is straightforward to
verify that $\cO_{\tld Y}(\lfloor (t - \epsilon')G - \epsilon \overline X\rfloor)$
is constant for sufficiently small positive $\epsilon'$ and $\epsilon$,
and therefore so is $\DuBois{X, \ba^{t - \epsilon'}}$.
Furthermore, because we are subtracting (and not adding)
$\epsilon'$ from $t$, we may in fact choose $\epsilon' = \epsilon$ (as
 long as they are both sufficiently small).  Therefore,
 we conflate the two $\epsilon$'s and write $\DuBois{X, \ba^{t - \epsilon}}$ to
 denote this object for $\epsilon=\epsilon'$ sufficiently small.

\begin{prop}
\label{LemmaLeftVanishing}
With the notation above,
$\myH^i \myD\left(\DuBois{X, \ba^{t - \epsilon}}\right) = 0$ for $i < -\dim X$.
\end{prop}
\begin{proof}
We first note that we may assume that no irreducible
component of $X$ is contained in $V(\ba)$.  To
see this, suppose $X = X_1 \cup X_2$ where $X_1$
is the union of those irreducible components of $X$ that are
not contained in $V(\ba)$ and $X_2$ is the union of the remaining components.  Then notice that
\[
\lfloor (t - \epsilon)G - \epsilon \overline X\rfloor = \lfloor (t - \epsilon)G -
\epsilon (\overline {X_1 \cup X_2}) \rfloor = \lfloor (t - \epsilon)G - \epsilon \overline {X_1} \rfloor
\]
since we choose $\epsilon$ arbitrarily small.  It follows
that $\DuBois{X, \ba^{t - \epsilon}} = \DuBois{X_1, \ba^{t - \epsilon}}$.
In particular, note that if $X \subset V(\ba)$, then $\DuBois{X, \ba^{t - \epsilon}} \qis 0$.
Therefore if $\myH^i \myD\left(\DuBois{X_1, \ba^{t - \epsilon}}\right) = 0$ for $i < -\dim X_1$,
then for all $i < -\dim X \leq -\dim X_1$, we have
$\myH^i \myD\left(\DuBois{X, \ba^{t - \epsilon}}\right) = 0$.

We proceed by induction on the dimension of $X$.
If $\dim X = 0$, then $X$ is disjoint from the support
$\ba$ and the result follows from the theory of Du
Bois singularities since $\myH^i \myD (\DuBois{X}) = 0$
for $i < -\dim X$, see \cite[Lemma 3.6]{KovacsSchwedeSmithLCImpliesDuBois}.

For the induction step, define $\Gamma \subset Y$
to be the reduced scheme
\[
\Gamma := (\Sing X) \cup V(\ba),
\]
where $\Sing X$ is the singular locus of $X$.
In particular $\Gamma$ contains $V(\ba)$.
Decompose $\Gamma = \Gamma_{\ba} \cup \Sigma_{\times}$
where $\Gamma_{\ba}$ is the union of components of $\Gamma$
that are contained in $V(\ba)$ and $\Sigma_{\times}$ is the
union of the components of $\Gamma$ that are not contained in $V(\ba)$.

Let $\pi : \tld Y \rightarrow Y$ be an embedded log
resolution of $X$ and log resolution of $(Y, \ba^{t - \epsilon})$
as in Lemma \ref{LemmaEmbeddedResolutionIsOk}.
Set $E_{\ba}$ to be the reduced pre-image of
$\Gamma_{\ba}$ in $\tld Y$, $E_{\times}$ to be
the reduced pre-image of $\Sigma_{\times}$ in
$\tld Y$, $\ba \cO_{\tld Y} = \cO_{\tld Y}(-G)$
and set $\tld X$ to be the strict transform of $X$.
We may assume that $\pi$ is an isomorphism outside of
$V(\ba)$ and $\Sigma_{\times}$.  Now write $\I_X \cO_{\tld Y}
= \I_{\tld X} \cO_{\tld Y}(-E)$ and note that
\begin{equation} \label{eqnEpsilonMeansTweaks}
\lfloor (t - \epsilon)G - \epsilon(E_{\times} + E_{\ba})
\rfloor = \lfloor (t - \epsilon)G - \epsilon E \rfloor  = \lfloor (t - \epsilon)G -
\epsilon E_{\times} \rfloor
\end{equation}
since we pick $\epsilon > 0$ to be arbitrarily small.

Notice we have the following short exact sequence.
\[
\Small
\xymatrix@C=10pt{
 0 \ar[r] & \cO_{\tld Y}(\lfloor (t - \epsilon)G -
 \epsilon E \rfloor) \tensor \I_{\tld X} \ar[r] &  \cO_{\tld Y}
 (\lfloor (t - \epsilon)G - \epsilon E_{\times} \rfloor)  \ar[r]
 & \cO_{\tld X}(\lfloor (t - \epsilon)G - \epsilon E_{\times} \rfloor|_{\tld X}) \ar[r] &  0.
 }
\]

By pushing forward, it follows that there exists the
following commutative diagram of exact triangles:
\[
\Small
\xymatrix@C=10pt@R=14pt{
\myR \pi_* \cO_{\tld X}(\lfloor (t - \epsilon)G -
\epsilon E_{\times} \rfloor|_{\tld X})[-1] \ar[d] \ar[r]&
 0  \ar[d] \ar[r] & \ar[d]  \myR \pi_* \cO_{\tld X}(\lfloor (t - \epsilon)G -
 \epsilon E_{\times} \rfloor|_{\tld X}) \ar[r]^-{+1} & \\
\myR \pi_* \cO_{\tld Y}(\lfloor (t - \epsilon)G - \epsilon E \rfloor) \tensor
\I_{\tld X}  \ar[d] \ar[r] & \myR \pi_* \cO_{\tld Y}(\lfloor (t - \epsilon)G \rfloor)
 \ar[d]^-{\sim} \ar[r] & \DuBois{X, \ba^{t-\epsilon}}\ar[d] \ar[r]^-{+1} & \\
\myR \pi_* \cO_{\tld Y}(\lfloor (t - \epsilon)G - \epsilon E_{\times} \rfloor)
\ar[d]^-{+1} \ar[r] & \ar[r] \ar[d]^-{+1} \myR \pi_* \cO_{\tld Y}
(\lfloor (t - \epsilon)G \rfloor) & \DuBois{\Sigma_{\times},
\ba^{t-\epsilon}} \ar[r]^{+1} \ar[d]^-{+1}& \\
& & &
}
\]
While in general, restricting divisors does not commute
with round-downs, in our case we do have $\lfloor (t - \epsilon)G -
\epsilon E_{\times} \rfloor|_{\tld X} = \lfloor (t - \epsilon)G|_{\tld X}
- \epsilon E_{\times} |_{\tld X} \rfloor$
because the divisors and the object we are restricting
them to are in simple normal crossings.
We dualize the right vertical column and obtain
\[
\xymatrix{
\myD\left(\DuBois{\Sigma_{\times}, \ba^{t-\epsilon}} \right)
\ar[r] & \myD\left(\DuBois{X, \ba^{t-\epsilon}} \right) \ar[r] &
\myR \pi_* \omega_{\tld X}^{\mydot}(\lceil - (t - \epsilon)G|_{\tld X}
+ \epsilon E_{\times}|_{\tld X} \rceil) \ar[r]^-{+1}  &. \\
}
\]
By taking cohomology, and using the inductive
hypothesis on $\myD\left(\DuBois{\Sigma_{\times}, \ba^{t-\epsilon}} \right)$,
we obtain our desired result.
\end{proof}

The following corollary is a key application of what we
have proven so far.  It allows us to relate the auxiliary
objects $\DuBois{X, \ba^{t-\epsilon}}$ with the maximal non-lc ideals $\mJ'(X; \ba^{t})$.

\begin{cor} \label{corDimDuboisIsNonLC}
Assume that $X$ is normal and equidimensional, $\ba$ is
an ideal sheaf on $X$ and that no component of $X$ is
contained inside $V(\ba)$.  Let $\pi' : \tld X \to X$ be a
log resolution, let $F$ be the exceptional divisor of $\pi$ and
set $\ba \cO_{\tld X} = \cO_{\tld X}(-H)$ then
\begin{equation} \label{eqDuBoisIsNonLC}
\myH^{-\dim X} \myD\left(\DuBois{X, \ba^{t - \epsilon}}\right)
= \pi'_* \cO_{\tld X}(\lceil K_{\tld X} - (t - \epsilon)H + \epsilon F \rceil)
\end{equation}
for all sufficiently small $\epsilon$ and \emph{any} embedding
of $X \subseteq Y$ into a smooth variety.  In particular,
$\myH^{-\dim X} \myD\left(\DuBois{X, \ba^{t - \epsilon}}\right)$
is independent of the choice of embedding of $X$ into $Y$.
\end{cor}
\begin{proof}
Since the right side is independent of the choice of resolution by
Lemma \ref{lem-a}, we assume that $\pi' := \pi|_{\tld X}$ is induced as
in Lemma  \ref{LemmaEmbeddedResolutionIsOk} and furthermore
that $\pi$ is an isomorphism outside of $V(\ba)$ and $\Sing X$.
The result follows immediately from the final exact
triangle used in the proof of  Proposition \ref{LemmaLeftVanishing} when
one notes that $\dim \Sigma_{\times} \leq \dim X - 2$ (since $X$ is normal).
Note that while there may be
components of $F$ and $H$ which coincide, choosing small
enough epsilon allows us to ignore such complication as in Equation (\ref{eqnEpsilonMeansTweaks}).
\end{proof}

In the case that $X$ is a complete intersection,
we will show that $\myH^{i} \myD\left(\DuBois{X, \ba^{t - \epsilon}}\right)
= 0$ for $i > -\dim X$.   First however, we need the following
lemma which will be a key point in an inductive
argument, the proof is similar to that of Proposition \ref{LemmaLeftVanishing}.
\begin{lem}
\label{LemmaHypersurfaceShortExactSequence}
Suppose that $Y$, $\ba$ and $X$ are as above and
suppose that no component of $X$ is contained inside $V(\ba)$.
 Suppose further that $X$ is a codimension 1
 subset of a reduced equidimensional scheme $Z \subseteq Y$
 also such that no component of $Z$ is contained inside $V(\ba)$.
 Let $\pi : \tld Y \rightarrow Y$ be a log resolution of
 $(Y, X \cup (\Sing Z), \ba^t)$ that is simultaneously
 an embedded resolution of $Z$ as
 in {\em{Lemma \ref{LemmaEmbeddedResolutionIsOk}}}.
 Furthermore, we assume that $\pi$ is an isomorphism outside of $(\Sing Z) \cup X \cup V(\ba)$.

Set $\tld Z$ to be the strict transform of $Z$,
set $\ba \cO_{\tld Y} = \cO_{\tld Y}(-G)$, set $\I_X \cO_{\tld Y} = \cO_{\tld Y}(-\overline X)$
and set $E_Z$ to be the divisor on $\tld Y$ such that
$\I_Z \cO_{\tld Y} = \I_{\tld Z} \cO_{\tld Y}(-E_Z)$.
Finally, let $\Sigma$ denote the union of components of
$\Sing Z$ which are not contained in $V(\ba)$.  Then there is an exact triangle
\[
\xymatrix{
\myR \pi_* \cO_{\tld Z}(\lfloor (t - \epsilon)G|_{\tld Z} -
\epsilon (E_Z + \overline X)|_{\tld Z} \rfloor) \ar[r]
 &  \DuBois{Z, \ba^{t-\epsilon}} \ar[r]
 & \DuBois{X \cup \Sigma, \ba^{t-\epsilon}} \ar[r]^-{+1} &.
}
\]
\end{lem}
\begin{proof}
We begin with the following short exact sequence:
\begin{align*}
 0 \to
\cO_{\tld Y}(\lfloor (t - \epsilon)G - \epsilon (E_Z + \overline{X}) \rfloor) \tensor \I_{\tld Z} &\to
\cO_{\tld Y}(\lfloor (t - \epsilon)G - \epsilon (E_Z + \overline{X}) \rfloor)  \\
&\to
\cO_{\tld Z}(\lfloor (t - \epsilon)G - \epsilon (E_Z + \overline{X}) \rfloor|_{\tld Z}) \to
 0.
\end{align*}
We set $E_{\times}$ to be the union
of the components of $\mathrm{Supp}(E_Z \cup\overline{X})$
whose images in $Y$ are contained in $X \cup \Sigma$
so that $\pi(E_\times) = X \cup \Sigma$.
Furthermore, notice that $\mathrm{Supp}(E_Z \cup \overline{X}) \setminus E_{\times}
\subseteq \Supp G$ so that
\[
\lfloor (t - \epsilon)G - \epsilon ( (E_Z)_{\red}
\cup \overline{X}_{\red} ) \rfloor = \lfloor (t - \epsilon)G -
\epsilon (E_Z + \overline{X}) \rfloor = \lfloor (t - \epsilon)G - \epsilon (E_{\times}) \rfloor
\]
again because $\epsilon$ is sufficiently small.

We now form a diagram with exact triangles
as columns and rows as in the proof of Proposition \ref{LemmaLeftVanishing}:
\[
\SMALL
\xymatrix@C=7pt@R=12pt{
\myR \pi_* \cO_{\tld Z}(\lfloor (t - \epsilon)G -
\epsilon (E_Z + \overline{X}) \rfloor|_{\tld Z})[-1] \ar[d] \ar[r]
& 0  \ar[d] \ar[r] & \ar[d]  \myR \pi_* \cO_{\tld Z}(\lfloor (t - \epsilon)G -
\epsilon (E_Z + \overline{X}) \rfloor|_{\tld Z}) \\
\myR \pi_* \cO_{\tld Y}(\lfloor (t - \epsilon)G -
\epsilon (E_Z + \overline{X}) \rfloor) \tensor \I_{\tld Z}
\ar[d] \ar[r] & \myR \pi_* \cO_{\tld Y}(\lfloor (t - \epsilon)G \rfloor)
\ar[d]^-{\sim} \ar[r] & \DuBois{Z, \ba^{t-\epsilon}}\ar[d] \\
\myR \pi_* \cO_{\tld Y}(\lfloor (t - \epsilon)G - \epsilon E_\times \rfloor)
 \ar[r] & \ar[r] \myR \pi_* \cO_{\tld Y}(\lfloor (t - \epsilon)G \rfloor)
 & \DuBois{X \cup \Sigma, \ba^{t-\epsilon}}
}
\]
Our desired exact triangle is the right vertical column.
\end{proof}

\begin{thm}\label{thmRightVanishing}
With the notation as in {\em{Proposition \ref{LemmaLeftVanishing}}},
suppose that $X$ is a complete intersection variety
in $Y$, and no component of $X$ is contained inside $V(\ba)$,
then $\myH^i \myD\left(\DuBois{X, \ba^{t - \epsilon}}\right)
= 0$ for $i > -\dim X$.
\end{thm}
\begin{proof}
We proceed by induction on the codimension
of $X$ in $Y$.  We begin with the case where $X$ is a hypersurface.
We have the following exact triangle:
\[
\xymatrix{
\myR \pi_* \cO_{\tld Y}(\lfloor (t - \epsilon)G - \epsilon \overline X\rfloor)
\ar[r] & \myR \pi_* \cO_{\tld Y}(\lfloor (t - \epsilon)G \rfloor) \ar[r] &
\DuBois{X, \ba^{t-\epsilon}} \ar[r]^-{+1} &,
}
\]
where $\overline X$ is the pullback of $X$ (if one
takes $\overline X$ to be the reduced pre-image of $X$,
you get the same result since $\epsilon$ is arbitrarily small).  Dualizing gives us:
\[
\xymatrix@C=15pt{
\myD\left(\DuBois{X, \ba^{t-\epsilon}}\right) \ar[r] & \myR
\pi_* \omega_{\tld Y}^{\mydot}(\lceil - (t - \epsilon)G \rceil) \ar[r]
&  \myR \pi_* \omega_{\tld Y}^{\mydot}(\lceil -(t - \epsilon)G + \epsilon \overline X\rceil) \ar[r]^-{+1} &.
}
\]
Taking cohomology gives us the claimed vanishing since
the cohomology of the right-most two terms vanish for $i > -\dim Y = -\dim X - 1$.
To see this explicitly, note that
the middle term vanishes due to
Kawamata--Viehweg vanishing in the form of local vanishing
for multiplier ideals.  The right most term also vanishes for the same reason once one notices that
\[
\myR \pi_* \omega_{\tld Y}^{\mydot}(\lceil -(t - \epsilon)G +
\epsilon \overline X\rceil) \tensor \cO_Y(-X)
\qis \myR \pi_* \omega_{\tld Y}^{\mydot}(\lceil -(t - \epsilon)G - (1 - \epsilon) \pi^* X \rceil).
\]

Now we assume that $X$ is a complete intersection in $Y$.
Choose hypersurfaces $H_1, \dots, H_{l}$ to be general hypersurfaces
containing $X$ such that $X = H_1 \cap \dots \cap H_l$.
Let $Z$ be an intersection of the first $l - 1 = (\dim Y - \dim X) - 1$
such hypersurfaces.  In this way, $X$ is a Cartier divisor in $Z$
and $Z \setminus X$ is smooth.  Notice also that $Z$ is certainly
$S_2$ and it is smooth at all points where $X$ is
smooth (which includes all the generic points of $X$).
In particular, $Z$ is smooth in codimension $1$
and thus it is normal.   Let $\pi : \tld Y \rightarrow Y$
be a log resolution of $(Y, X, \ba^t)$ that is simultaneously
a factorizing embedded resolution of $Z$, as in
Lemma \ref{LemmaHypersurfaceShortExactSequence}.
Dualizing the triangle from
Lemma \ref{LemmaHypersurfaceShortExactSequence}, we obtain a triangle:
\begin{equation}\label{eqDualedTriangleInCIProof}
\xymatrix@C=15pt{
\myD\left(\DuBois{X, \ba^{t-\epsilon}} \right)\ar[r]
&  \myD\left(\DuBois{Z, \ba^{t-\epsilon}} \right)\ar[r]
& \myR \pi_* \omega_{\tld Z}^{\mydot}(\lceil -(t - \epsilon)
G|_{\tld Z} + \epsilon (E_Z + \overline X) \rceil|_{\tld Z})\ar[r]^-{+1} &.
}
\end{equation}
Since $Z \setminus X$ is smooth, observe that
\begin{align*}
&\myH^{i} \myR \pi_* \omega_{\tld Z}^{\mydot}
(\lceil -(t - \epsilon)G + \epsilon (E_Z + \overline X) \rceil |_{\tld Z}) \tensor \cO_Z(-X) \\
 =&\myH^{i} \myR \pi_* \omega_{\tld Z}^{\mydot}
 (\lceil -(t - \epsilon)G|_{\tld Z} + \epsilon (\overline X)|_{\tld Z} \rceil ) \tensor \cO_Z(-X)
\end{align*}
which vanishes for $i > -\dim Z$ using the projection formula
and local vanishing for multiplier ideals.  Furthermore,
$\myH^i \myD\left(\DuBois{Z, \ba^{t - \epsilon}}\right) = 0$ for $i > -(\dim X + 1) = -\dim Z$
by the induction hypothesis.  Thus taking
cohomology of Equation (\ref{eqDualedTriangleInCIProof}) for $i > -\dim Z$ gives us the desired result.
\end{proof}

\begin{cor}
\label{CorollaryFirstShortExactSequence}
If $Z$ is a normal complete intersection,
$X$ is a Weil divisor in $Z$ and $V(\ba)$ doesn't
contain any component of $X$ or $Z$, then there is a short exact sequence
\[
\begin{array}{rcl}
0 &\to& \myH^{-\dim Z} \myD \left(\DuBois{Z, \ba^{t-\epsilon}} \right) \\
&\to&  \pi_* \cO_{\tld Z}(\lceil K_{\tld Z} -
(t - \epsilon)G|_{\tld Z} + \epsilon (E_Z + \overline X )|_{\tld Z} \rceil) \\
&\to&  \myH^{-\dim X} \myD \left( \DuBois{X \cup \Sigma, \ba^{t-\epsilon}} \right) \\
& \to & 0
\end{array}
\]
where $\Sigma$ and the remaining notation comes from
{\em{Lemma \ref{LemmaHypersurfaceShortExactSequence}}}.
\end{cor}
\begin{proof}
Simply dualize the sequence from
Lemma \ref{LemmaHypersurfaceShortExactSequence}.
Then take cohomology and apply the vanishing results
Proposition \ref{LemmaLeftVanishing} and Theorem \ref{thmRightVanishing}.
\end{proof}

\begin{lem} \label{lemCanRemoveExtraneousBitsFromDualizing}
With the notation from {\em{Corollary \ref{CorollaryFirstShortExactSequence}}},
further assume that $X$ is normal and Cartier.
Then $\myH^{-\dim X} \myD \left( \DuBois{X \cup \Sigma, \ba^{t-\epsilon}} \right)
\cong \myH^{-\dim X} \myD \left( \DuBois{X, \ba^{t-\epsilon}} \right)$.
In particular, we have a short exact sequence
\[
\begin{array}{rcl}
0 &\to & \myH^{-\dim Z} \myD
\left(\DuBois{Z, \ba^{t-\epsilon}} \right) \\
 &\to & \pi_* \cO_{\tld Z}(\lceil K_{\tld Z} - (t - \epsilon)G |_{\tld Z} +
 \epsilon (E_Z  + \overline X)|_{\tld Z} \rceil) \\
  &\to &\myH^{-\dim X} \myD \left( \DuBois{X, \ba^{t-\epsilon}} \right) \\
  &\to & 0
\end{array}
\]
\end{lem}
\begin{proof}
We will need the following:
\begin{claim} \label{ClaimTriangleExists}  There is a triangle
\[
\xymatrix{
 \DuBois{X \cup \Sigma, \ba^{t - \epsilon}} \ar[r]
 & \DuBois{X, \ba^{t-\epsilon}} \oplus \DuBois{\Sigma, \ba^{t - \epsilon}}
 \ar[r] & \DuBois{\Sigma \cap X, \ba^{t - \epsilon}} \ar[r]^-{+1} &.
}
\]
\end{claim}
\begin{proof}[Proof of {\em{Claim \ref{ClaimTriangleExists}}}]
We fix $\pi : \tld Y \to Y$ to be an embedded resolution of $Z$
which is also a log resolution of $X$, $\Sigma$,
$X \cap \Sigma$ and $\ba$ and write $\ba \cdot \cO_{\tld Y}
= \cO_{\tld Y}(-G)$.  Let $E_1$ (respectively $E_2$) denote the
reduced pre-image of $X$ in $\tld Y$ (respectively, of $\Sigma$ in $\tld Y$)
and note that $E_1 \cap E_2$ is the reduced pre-image of
$X \cap \Sigma$ (which we assumed was also a divisor).  Furthermore, let
\begin{itemize}
 \item{} $E_1'$ denote the union of the components
 of $E_1$ which are not components of $G$, and let
 \item{} $E_2'$ denote the union of the components
 of $E_2$ which are not components of $G$.
\end{itemize}
Consider $E_1' \cap E_2'$.  I claim that this is a divisor and thus is
equal to the union of the components of $E_1 \cap E_2$ which
are not contained in $G$.  Suppose that $L$ is an irreducible component of
$E_1' \cap E_2'$ and suppose $L$ is not a divisor.  On the other hand,
$L \subseteq \pi^{-1}(X \cap \Sigma)$
by construction, and so it must be contained in some divisor $H$ lying over $X \cap \Sigma$.
Since $L$ is an irreducible component of $E_1' \cap E_2'$,
we must have that $H$ is not a component of $E_1' \cap E_2'$, and so in particular,
$H$ is contained in $G$.  Therefore $L$ is contained in $G$.

Since $L$ is not a divisor, it must be
codimension 2 (since it is a component of the the intersection of two divisors in a smooth space).
Suppose $L$ is in the intersection of a component $F_1 \subseteq E_1'$ and
a component $F_2$ of $E_2'$ (neither $F_1$ or $F_2$ are in $G$ by construction).
But then $F_1$, $F_2$ and $H$ would not be in simple normal
crossings at the generic point of $L$ (because the generic point of $L$
is a two dimensional regular local ring).  Therefore $L$ is not contained in $G$, a contradiction.

It follows that we have the following short exact sequence
\[
\hspace*{-0.2em}
\xymatrix@C=10pt{
0 \ar[r] & \cO_{\tld Y}( -(E_1' \cup E_2') ) \ar@{=}[d]
  \ar[r] & \cO_{\tld Y}(-E_1') \oplus \cO_{\tld Y}(- E_2') \ar@{=}[d]
  \ar[r]^-{-} & \cO_{\tld Y}(- (E_1' \cap E_2'))\ar@{=}[d] \ar[r]  & 0
\\
 0 \ar[r] & \cO_{\tld Y}(\lfloor -\epsilon (E_1' \cup E_2') \rfloor)
  \ar[r] & \cO_{\tld Y}(\lfloor  -\epsilon E_1' \rfloor)
  \oplus \cO_{\tld Y}(\lfloor -\epsilon E_2' \rfloor)
  \ar[r]^-{-} & \cO_{\tld Y}(\lfloor - \epsilon (E_1' \cap E_2') \rfloor) \ar[r]  & 0
}
\]
where the third horizontal map sends $(a,b)$ to $a-b$.
We tensor this short exact sequence with
$\cO_{\tld Y}(\lfloor (t - \epsilon)G\rfloor)$ and then map the result into
\[
\xymatrix@C=10pt{
 0 \ar[r] & \cO_{\tld Y}(\lfloor (t - \epsilon)G \rfloor)
  \ar[r] & \cO_{\tld Y}(\lfloor (t - \epsilon)G \rfloor)
  \oplus \cO_{\tld Y}(\lfloor (t - \epsilon)G\rfloor)
  \ar[r]^-{-} & \cO_{\tld Y}(\lfloor (t - \epsilon)G  \rfloor) \ar[r] & 0.
}
\]
By taking the cokernel, we obtain a short exact sequence from the nine-lemma,
\[
{
\xymatrix@C=12pt{
 0 \ar[r] & M_{X \cup \Sigma, \ba^{t-\epsilon}}
  \ar[r] & M_{X, \ba^{t-\epsilon}} \oplus M_{\Sigma, \ba^{t-\epsilon}}
  \ar[r]^-{-} & M_{X \cap \Sigma, \ba^{t-\epsilon}} \ar[r] & 0.
}
}
\]
Pushing forward completes the proof of the claim.
\end{proof}
Dualizing the triangle from the claim, we obtain
\[
 \xymatrix@C=15pt{
 \myD\left(\DuBois{\Sigma \cap X, \ba^{t - \epsilon}}\right)
 \ar[r] & \myD\left(\DuBois{X, \ba^{t - \epsilon}}\right)
 \oplus \myD\left(\DuBois{\Sigma, \ba^{t - \epsilon}}\right)
 \ar[r] & \myD\left(\DuBois{X \cup \Sigma, \ba^{t - \epsilon}}\right) \ar[r]^-{+1} &.
}
\]
We then take cohomology.
Notice that $\Sigma \cap X \subseteq (\Sing Z) \cap X \subseteq \Sing X$
since $X$ is Cartier.  The result follows by Proposition \ref{LemmaLeftVanishing}
since $\dim (\Sigma \cap X) \leq \dim X - 2$ (since $X$ is normal)
and $\dim \Sigma \leq \dim Z - 2 = \dim X - 1$ (since $Z$ is normal).
\end{proof}

When we combine
Lemma \ref{lemCanRemoveExtraneousBitsFromDualizing}
with Corollary \ref{corDimDuboisIsNonLC}, we obtain

\begin{cor}
\label{CorRealExactSequence}
 Suppose that $X$ is a normal irreducible Cartier
 divisor in $Z$ which itself is a normal irreducible complete
 intersection and suppose that $\ba$ is a non-zero ideal
 sheaf on $Z$ such that $V(\ba)$ does not contain $X$.
 Let $\pi : \tld Z \to Z$ be a log resolution of $Z$, $X$ and
 $\ba$, and let $\tld X$ denote the strict transform of $X$.
 Set $E$ to be the exceptional set of $\pi$, $F$ to be the exceptional set
 of $(\pi|_{\tld X})$, and write $\ba \cdot \cO_{\tld Z} = \cO_{\tld Z}(-G)$.
 Then we have a short exact sequence
\[
\begin{array}{rcl}
0 &\to & \pi_* \cO_{\tld Z}(\lceil K_{\tld Z} - (t - \epsilon)G + \epsilon E \rceil) \\
&\to & \pi_* \cO_{\tld Z}(\lceil K_{\tld Z} - (t - \epsilon)G + \epsilon (E + \pi^* X) \rceil) \\
&\to & \pi_* \cO_{\tld X}(\lceil K_{\tld X} - (t - \epsilon)G|_{\tld X} + \epsilon F \rceil) \\
&\to & 0
\end{array}
\]
\end{cor}
\begin{proof}
 This follows from an application of
 Corollary \ref{corDimDuboisIsNonLC} to the short exact sequence of Lemma \ref{lemCanRemoveExtraneousBitsFromDualizing}.
\end{proof}

\begin{thm}
\label{thmRestrictionForMaximalNonLC}
 If $Z$ is a normal complete intersection,
 $X \subset Z$ is a normal Cartier divisor and $\ba$
 is an ideal sheaf on $Z$ such that $V(\ba)$ does
 not contain any component of $Z$ or $X$, then there is a short exact sequence{\em{:}}
\[
\xymatrix@C=12pt{
0 \ar[r] & \mJ'(Z; \ba^{t}) \tensor \cO_Z(-X)
\ar[r] & \ar[r] \mJ'( (Z, X); \ba^{t}) \ar[r] & \mJ'(X; \ba^{t}) \ar[r] & 0.
}
\]
\end{thm}
\begin{proof}
Tensor the exact sequence from
Corollary \ref{CorRealExactSequence} with $\cO_Z(-K_Z - X)$ noticing that
\[
\begin{array}{rcl}
& & \pi_* \cO_{\tld X}
(\lceil K_{\tld X} - (t - \epsilon)G|_{\tld X} + \epsilon F \rceil) \tensor \cO_Z(-K_Z - X)
\\
 & \cong & \pi_* \cO_{\tld X}(\lceil K_{\tld X} -
 (t - \epsilon)G|_{\tld X} + \epsilon F \rceil) \tensor \cO_X \tensor \cO_Z(-K_Z - X)  \\
 & \cong & \pi_* \cO_{\tld X}(\lceil K_{\tld X} -
 (t - \epsilon)G|_{\tld X} + \epsilon F \rceil) \tensor \cO_X(-K_X) \\
 & \subseteq & \cO_X.
\end{array}
\]
\end{proof}

\part{A characteristic $p$ analog of maximal non-lc ideals}
This part is devoted to the study of a positive characteristic analog of the maximal non-lc ideal sheaves.
As mentioned earlier, this is independent of the previous part, except for the definition of the maximal non-lc ideal sheaves.  We should also mention that this is a first attempt.  The authors expect that further refinements of the definition may be necessary, in particular see Remark \ref{remOtherDefinitions}.

\section{Non-F-pure ideals}\label{secNonFpureIdeals}
In this section, we introduce a characteristic $p$ analog of maximal non-lc ideals, called non-F-pure ideals, and study their basic properties.

From this point forward, all rings are Noetherian commutative rings with identity.
For a reduced ring $R$, we denote by $R^{\circ}$ the set of elements of $R$ that are not in any minimal prime.

Let $R$ be a reduced ring of characteristic $p>0$.
For an ideal $I$ of $R$ and a power $q$ of $p$, we denote by $I^{[q]}$ the ideal generated by the $q^{\rm th}$ powers of elements of $I$.
Given an $R$-module $M$ and an integer $e \ge 1$, we will use $F^e_*M$ to denote the $R$-module which agrees with $M$ as an additive group, but where the multiplication is defined by $r \cdot m=r^{p^e}m$.
For example, $I \cdot F^e_*M \cong F^e_*(I^{[p^e]}M)$.
We say that $R$ is \textit{$F$-finite} if $F^1_*R$ is a finitely generated $R$-module. For example,
a field $k$ is $F$-finite if and only if $[k:k^p] < \infty$, and any algebra essentially of finite type over an $F$-finite field is $F$-finite.

\begin{defn}
A \textit{triple} $(R, \Delta, \ba^t)$ is the combined information of
\begin{enumerate}
\item[(i)] an $F$-finite reduced ring $R$,
\item[(ii)] an ideal $\ba \subseteq R$ such that $\ba \cap R^{\circ} \ne \emptyset$,
\item[(iii)] a real number $t>0$.
\end{enumerate}
Furthermore, if $R$ is a normal domain, then we also consider
\begin{enumerate}
\item[(iv)] an effective $\R$-divisor $\Delta$ on $X=\Spec R$.
\end{enumerate}
If $\ba = R$ (resp., $\Delta=0$) then we call the triple $(R, \Delta, \ba^t)$ a \textit{pair} and denote it by $(R, \Delta)$ (resp., $(R, \ba^t)$).
\end{defn}

First we recall the definitions of generalized test ideals, F-purity and F-regularity for triples.
\begin{defn}\label{testidealdef}
Let $(R, \Delta, \ba^t)$ be a triple.
\begin{enumerate}
\item[(i)] (cf. \cite[Definition-Proposition 3.3]{BSTZ}, \cite[Lemma 2.1]{HT}, \cite[Definition-Theorem 6.5]{HY}, \cite[Definition 2.6]{Ta})
The \textit{big generalized test ideal} $\tau_b((R, \Delta);\ba^t)$ is defined to be
$$\tau_b((R, \Delta);\ba^t)=\sum_{e \ge n}\sum_{\varphi} \varphi(F^e_*(\ba^{\lceil t p^e \rceil}d)),$$
where $n$ is an arbitrary positive integer and $\varphi$ ranges over all elements of $\Hom_R(F^e_*R(\lceil p^e \Delta \rceil), R) \subseteq \Hom_R(F^e_*R, R)$, and where $d \in R^{\circ}$ is a big test element for $R$.
We do not give the definition of big test elements here (see \cite{Ho} for the definition of big test elements), but, for example, if the localized ring $R_d$ is regular, then some power $d^n$ is a big test element for $R$ by \cite[p.63, Theorem]{Ho}.

\item[(ii)] (\cite[Definition 3.2]{Sc4} cf.~\cite[Definition 2.1]{HW}, \cite[Definition 3.1]{SchwedeSharpFpurity}, \cite[Definition 3.1]{Ta2})
Let $\overline{\ba^{\lceil t \bullet \rceil}}:=\{\ba_m\}_{m \in \N}$ be a graded family of ideals defined by $\ba_m=\overline{\ba^{\lceil t(m-1) \rceil}}$, where $\overline{\mathfrak{b}}$ is the integral closure of an ideal $\mathfrak{b} \subseteq R$. 
We say that $(R, \Delta, \overline{\ba^{\lceil t \bullet \rceil}})$ is \textit{sharply F-pure} if there exist an integer $e \ge 1$ and a map
$$\varphi \in \Hom_R(F^e_*R(\lceil (p^e-1)\Delta \rceil), R) \cdot F^e_*\ba_{p^e-1}$$
such that $\varphi(1)=1$.

\item[(iii)] (\cite[Definition 3.2]{Sc4} cf.~\cite[Definition 2.1]{HW}, \cite[Definition 3.1]{Ta2})
We say that $(R, \Delta, \ba^t)$ is \textit{strongly F-regular} if for every $d \in R^{\circ}$, there exist an integer $e \ge 1$ and a map
$$\varphi \in \Hom_R(F^e_*R(\lceil (p^e-1)\Delta \rceil), R) \cdot F^e_*(\ba^{\lceil t(p^e-1) \rceil}d)$$
such that $\varphi(1)=1$.
\end{enumerate}
\end{defn}

\begin{rem}
(1) 
Considering the case where $d=1$, one can easily see that if $(R, \Delta, \ba^t)$ is strongly F-regular, then $(R, \Delta, \overline{\ba^{\lceil t \bullet \rceil}})$ is sharply F-pure. Also, it follows from \cite[Corollary 5.7]{Sc4} that $(R, \Delta, \ba^t)$ is strongly F-regular if and only if $\tau_b((R, \Delta);\ba^t)=R$.

(2) In the case where $\Delta=0$ and $\ba=0$, $(R, \Delta, \overline{\ba^{\lceil t \bullet \rceil}})$ is sharply F-pure (resp. $(R, \Delta, \ba^t)$ is strongly F-regular) if and only if the ring $R$ is F-pure (resp. strongly F-regular), and $\tau_b((R, \Delta); \ba^t)=\tau_b(R)$ is the big test ideal.
Refer to \cite{HR}, \cite{HH} and \cite{Ho} for basic properties of F-pure rings, strongly F-regular rings and the big test ideal, respectively.
\end{rem}

Thanks to \cite{HY} and \cite{Ta}, the generalized test ideal can be viewed as a characteristic $p$ analog of the multiplier ideal. In particular, a strongly F-regular triple $(X=\Spec R, \Delta, \ba^t)$ corresponds to a klt triple under the condition that $K_X+\Delta$ is $\Q$-Cartier.
Also, a sharply F-pure triple is expected to correspond to a lc triple under the same condition.
Employing this philosophy, we introduce a characteristic $p$ analog of maximal non-lc ideals.

\begin{defn}[\textup{cf.~\cite[Corollary 2.14]{Bl}}]
Let $(R, \Delta, \ba^t)$ be a triple.
We denote the integral closure of an ideal $\mathfrak{b}$ of $R$ by $\overline{\mathfrak{b}}$. 
We then define the family of  ideals $\{\sigma_n((R,\Delta);\ba^t)\}_{n \in \N}$ inductively as follows\footnote{This definition of $\sigma_n((R, \Delta); \ba^t)$ is slightly different from the one given in a  previous version. We changed the definition because the proof of Theorem \ref{monomialcharp} didn't work for the previous one.}:
\begin{align*}
\sigma_1((R,\Delta);\ba^t)&=\sum_{e \ge 1}\sum_{\varphi} \varphi(F^e_*\overline{\ba^{\lceil t (p^e-1) \rceil}}),\\
\sigma_2((R,\Delta);\ba^t)&=\sum_{e \ge 1}\sum_{\varphi} \varphi(F^e_*(\sigma_1((R, \Delta);\ba^t)\overline{\ba^{\lceil t (p^e-1) \rceil}})),\\
\sigma_n((R,\Delta);\ba^t)&=\sum_{e \ge 1}\sum_{\varphi} \varphi(F^e_*(\sigma_{n-1}((R, \Delta); \ba^t)\overline{\ba^{\lceil t (p^e-1) \rceil}})),
\end{align*}
where
$\varphi$ runs through all elements of $\Hom_R(F^e_*R(\lceil (p^e-1)\Delta \rceil), R) \subseteq \Hom_R(F^e_*R, R)$.
Just for the convenience, we decree that $\sigma_0((R,\Delta);\ba^t)=R$.
It follows from \cite[Proposition 2.13]{Bl} that the descending chain
$$R \supseteq \sigma_1((R,\Delta);\ba^t) \supseteq \sigma_2((R,\Delta);\ba^t) \supseteq \sigma_3((R,\Delta);\ba^t) \supseteq \cdots$$
stabilizes at some $n \in \N$.
Then the \textit{non-F-pure ideal} $\sigma((R,\Delta);\ba^t)$ is defined to be
$$\sigma((R,\Delta);\ba^t)=\sigma_n((R,\Delta);\ba^t)=\sigma_{n+1}((R,\Delta);\ba^t)=\cdots.$$
When $\ba=R$ (resp. $\Delta=0$), we simply denote this ideal by $\sigma(R, \Delta)$ (resp. $\sigma(R, \ba^t)$).
\end{defn}

\begin{rem}\label{remonexponent}
(1)
If $\Delta = 0$ and $\ba = R$, then $\sigma((R, \Delta); \ba^t) = \sigma(R)$ is just the image of the map $\Hom_R(F^e_* R, R) \to R$ which sends $\phi$ to $\phi(1)$, at least for $e \gg 0$.  In the case that $R$ is local and Gorenstein with perfect residue field, this image is stable for $e \gg 0$ by the Matlis dual of a celebrated result of Hartshorne and Speiser, see \cite[Proposition 1.11]{HS}.  Thus, Blickle's stabilization result \cite[Proposition 2.13]{Bl} should be viewed as a large generalization of Hartshorne and Speiser's result.

(2) The real exponent $t$ on the ideal $\ba$ in the definition of big generalized test ideals (Definition \ref{testidealdef} (i)) is just a formal notation, but it is compatible with ``real" powers of $\ba$.
That is, setting $\ib=\ba^n$ for $n \in \N$, one has $\tau_b((R, \Delta); \ib^t)=\tau_b((R, \Delta); \ba^{nt})$.
However, in the case of non-F-pure ideals, it is not compatible in general.
For example, if $R=k[x]$ is the one-dimensional polynomial ring over an $F$-finite field $k$ of characteristic $p>0$, then $\sigma(R, (x))=R$ but $\sigma(R, (x^p)^{1/p})=x$.
\end{rem}

\begin{rem}
Classically, the test ideal is defined as the annihilator ideal of some submodule of the injective hull.
We can define the non-F-pure ideal in a similar way.

Let $(R, \m)$ be an $F$-finite reduced local ring of characteristic $p>0$, and denote by $E=E_R(R/\m)$ the injective hull of the residue field $R/\m$.
For each integer $e \ge 1$, we denote $\F^{e, \Delta}(E):= F^e_*R(\lceil (p^e-1)\Delta \rceil) \otimes_R E$ and regard it as an $R$-module by the action of $F^e_*R \cong R$ from the left.
Then the $e^{\rm th}$ iteration of the Frobenius map induces a map $F^e:E  \to \F^{e, \Delta}(E)$.
The image of $z \in E$ via this map is denoted by $z^{p^e}:=1 \otimes z=F^e(z) \in \F^{e, \Delta}(E)$.

$N_1((R,\Delta);\ba^t)$ is defined to be the submodule of $E$ consisting of all elements $z \in E$ such that $\overline{\ba^{\lceil t(p^e-1) \rceil}}z^{p^e}=0$ in $\F^{e, \Delta}(E)$ for all $e \in \N$.
$N_2((R,\Delta);\ba^t)$ is defined to be the submodule of $E$ consisting of all elements $z \in E$ such that $\mathrm{Ann}_R(N_1((R,\Delta);\ba^t))\overline{\ba^{\lceil t(p^e-1) \rceil}}z^{p^e}=0$ in $\F^{e, \Delta}(E)$ for all $e \in \N$.
Inductively we define $N_n((R,\Delta);\ba^t)$ to be the submodule of $E$ consisting of all elements $z \in E$ such that $\mathrm{Ann}_R(N_{n-1}((R,\Delta);\ba^t))\overline{\ba^{\lceil t(p^e-1) \rceil}}z^{p^e}=0$ in $\F^{e, \Delta}(E)$ for all $e \in \N$.
Then the ascending chain
$$N_1((R,\Delta);\ba^t) \subseteq N_2((R,\Delta);\ba^t) \subseteq N_3((R,\Delta);\ba^t) \subseteq \cdots \subseteq E$$
stabilizes at some $n \in \N$ and $\sigma((R, \Delta);\ba^t)=\mathrm{Ann}_R(N_n((R,\Delta);\ba^t))$.
\end{rem}

\begin{rem}
\label{remOtherDefinitions}
There are numerous other non-$F$-pure ideals that one can define.  The ideal $\sigma((R, \Delta); \ba^t)$ is the largest ideal that we know of that both commutes with localization and seems naturally determined.

We briefly enumerate some of the other potential non-$F$-pure ideals and describe some of their advantages and disadvantages.  All of these ideals define the non-$F$-pure locus of $(R, \Delta, \ba^t)$ (if they stabilize in the right way).
\begin{itemize}
\item[(1)]  $\sigma_1((R,\Delta);\ba^t)$.  This ideal still defines the non-$F$-pure locus and its formation commutes with localization.  However, we do not believe 
that Theorem \ref{Fpure=>lc} holds for this ideal.  
The same comments also hold for the other $\sigma_i$.

\item[(2)]  For any fixed $n$, consider the ideal $\sigma_n'((R, \Delta); \ba^t) = \sum_{e \ge n}\sum_{\varphi} \varphi(F^e_*\overline{\ba^{\lceil t (p^e-1) \rceil}})$.  
This ideal is non-canonically determined,  because of the choice of $n$ (for various reasons, it is desirable to choose $n$ sufficiently large).
It may be that $\sigma'((R, \Delta), \ba^t) := \cap \sigma_n'((R, \Delta); \ba^t)$ is a good alternative, but we do not know if this intersection stabilizes for sufficiently large $n$ (if it does, then $\sigma'((R, \Delta, \ba^t)$ also commutes with localization).
If it does, then the ideal $\sigma'((R, \Delta); \ba^t)$ defines the non-$F$-pure locus and is {\it a priori} larger than $\sigma((R, \Delta); \ba^t)$, see Lemma \ref{nonFpurecontainment} below.  If $R$ is local with injective hull $E$ of the residue field, $\Delta = 0$ and $\ba = R$, then $\sigma'((R, \Delta);\ba^t)=\sigma'(R)$ also coincides with $$\Ann_R 0^F_{E} = \Ann_R \{ z \in E\,|\, z^{p^e} = 0 \text{ for }e \gg 0\},$$ the annihilator ideal of the Frobenius closure $0^F_{E}$ of the zero submodule in the injective hull $E$.

\item[(3)]  Suppose now that $R$ is local and that $(p^{e_0} - 1)(K_R + \Delta)$ is Cartier (for some $e_0$).  Then consider the ideal $\sum_{e = n e_0, n > 0}\sum_{\varphi} \varphi(F^e_*\overline{\ba^{\lceil t (p^e-1) \rceil}})$.  This ideal suffers from the same issue that (2) does, but it is even smaller.  If it can be shown to stabilize for sufficiently large and divisible $e_0$, then it would be useful.  In particular, one could prove versions of the restriction theorem (Theorem \ref{restrictionTheoremForNonFPure}) for triples $(R, \Delta, \ba^t)$.

\item[(4)]  Associated to a triple $(R, \Delta, \ba^t)$, one can define a \emph{Cartier-algebra} on $R$, see \cite{Bl} for the definition and details.  The ideal $\sigma((R, \Delta); \ba^t)$ is a natural object associated to this algebra.  If one replaces this Cartier-algebra by a Veronese sub-algebra, one obtains a different non-$F$-pure ideal.  If there is some stabilization of these non-$F$-pure ideals for sufficiently fine Veronese sub-algebras, then this could be very useful.
\end{itemize}

Suppose that $R$ is $\Q$-Gorenstein with index not divisible by $p > 0$.  Further suppose that $\Delta = 0$ and $\ba = R$.  Then $\sigma(R)$ coincides with $\sigma'_n(R)$ from (2) for $n \gg 0$.  This ideal also coincides with the ideal from (3) for sufficiently large and divisible $e_0$ and coincides with ideal from (4) for a sufficiently fine Veronese subalgebra.

To see that $\sigma'_n(R) = \sigma(R)$ for $n \gg 0$, first assume without loss of generality that $R$ is local.
Then, notice that the evaluation-at-1 map $\Hom_R(F^{e+1}_* R, R) \to R$ factors through $\Hom_R(F^e_* R, R)$.  Thus $\sigma'_n(R)$ is simply the image of $\Hom_R(F^{n}_* R, R) \to R$.  If $R$ is $\Q$-Gorenstein and $(p^{e_0} - 1)K_R$ is Cartier, then it follows from \cite[Corollary 3.10]{Sc2} that the image of $\Hom_R(F^{me_0}_* R, R) \to R$ is contained in $\sigma_m(R)$.  This implies that $\sigma'_n(R) \subseteq \sigma(R)$ for $n \gg 0$.  The reverse containment is done (in much greater generality) in Lemma \ref{nonFpurecontainment} below.

The equality with the ideals from (3) and (4) follow similarly.
\end{rem}

Before discussing the basic properties of non-F-pure ideals, we start with the following technical lemma.

\begin{lem}\label{nonFpurecontainment}
Let $(R, \Delta, \ba^t)$ be a triple.
Then for each $n \in \N$ and $i \in \Z_{\ge 0}$, one has
$$\sigma_{n+i}((R,\Delta);\ba^t) \subseteq \sum_{e \ge i+1}\sum_{\varphi} \varphi(F^e_*(\sigma_{n-1}((R, \Delta); \ba^t)\overline{\ba^{\lceil t (p^e-1) \rceil}}))$$
where $\varphi$ ranges over all elements of $\Hom_R(F^e_*R(\lceil (p^e-1)\Delta \rceil), R)$.
In particular, for all $n \in \N$, the ideal $\sigma_n((R, \Delta); \ba^t)$ is contained in
$$\sum_{e \ge n}\sum_{\varphi} \varphi(F^e_*\overline{\ba^{\lceil t (p^e-1) \rceil}}) = \sigma_n'((R, \Delta); \ba^t).$$
\end{lem}

\begin{proof}
From the point of view of Blickle's theory of Cartier algebras, this statement is essentially obvious.  We write down a proof in detail however.

 For all integers $i \ge 0$, set
$$\sigma_{i, n}((R, \Delta);\ba^t):=\sum_{e \ge i+1}\sum_{\varphi} \varphi(F^e_*(\sigma_{n-1}((R, \Delta); \ba^t)\overline{\ba^{\lceil t (p^e-1) \rceil}})).$$
We will prove the assertion by induction on $i$.
Obviously we may assume that $i \ge 1$.
Let $e \ge 1$ be an integer, and fix any $\varphi \in \Hom_R(F^e_*R(\lceil (p^e-1)\Delta \rceil), R)$, $a \in F^e_*\overline{\ba^{\lceil t(p^e-1) \rceil}}$ and $b \in F^e_*\sigma_{n+i-1}((R, \Delta); \ba^t)$.
It then suffices to show that $\varphi(ab) \in \sigma_{i, n}((R, \Delta);\ba^t)$.
It follows from the induction hypothesis that $b \in F^e_*\sigma_{i-1, n}((R, \Delta);\ba^t)$, that is,
$$b \in \sum_{l \ge i}\sum_{\psi} \psi(F^{e+l}_*(\sigma_{n-1}((R, \Delta); \ba^t)\overline{\ba^{\lceil t (p^l-1) \rceil}})),$$
where $\psi$ ranges over all elements of $\Hom_{F^e_*R}(F^{e+l}_*R(\lceil (p^l-1)\Delta \rceil), F^e_*R)$.
Since $p^l \lceil (p^e-1) \Delta \rceil+\lceil (p^l-1) \Delta \rceil \geq \lceil (p^{e+l}-1) \Delta \rceil$,
the map $\varphi \circ \psi$ can be regarded as an element of $\Hom_R(F^{e+l}_*R(\lceil (p^{e+l}-1)\Delta \rceil), R)$.
Also, since $p^l\lceil t(p^e-1) \rceil+\lceil t(p^l-1) \rceil \geq \lceil t(p^{e+l}-1) \rceil$, we have $(F^e_*\overline{\ba^{\lceil t(p^e-1) \rceil}})(F^{e+l}_*\overline{\ba^{\lceil t (p^l-1) \rceil}}) \subseteq F^{e+l}_* \overline{\ba^{\lceil t(p^{e+l}-1) \rceil}}$.
Thus,
\begin{align*}
\varphi(ab) \in & \sum_{l \ge i}\sum_{\psi} (\varphi \circ \psi)(F^e_*\overline{\ba^{\lceil t(p^e-1) \rceil}} F^{e+l}_*(\sigma_{n-1}((R, \Delta); \ba^t)\overline{\ba^{\lceil t (p^l-1) \rceil}}))\\
\subseteq & \sum_{l \ge i}\sum_{\psi} (\varphi \circ \psi)(F^{e+l}_*(\sigma_{n-1}((R, \Delta); \ba^t)\overline{\ba^{\lceil t (p^{e+l}-1) \rceil}}))\\
\subseteq & \sigma_{i, n}((R, \Delta);\ba^t).
\end{align*}
\end{proof}

\begin{que}
Let the notation be the same as in Lemma \ref{nonFpurecontainment}.
If $n$ is sufficiently large, then does $\sigma((R, \Delta); \ba^t)$ coincide with $\sigma_n'((R, \Delta); \ba^t)$?
\end{que}

We list basic properties of non-F-pure ideals.
\begin{prop}\label{basicnonFpureideal}
Let $(R, \Delta, \ba^t)$ be a triple.
\begin{enumerate}
\item
For any effective $\R$-divisor $\Delta' \le \Delta$ on $\Spec R$, for any ideal $\ba \subseteq \ib \subseteq R$ and for any real number $s \le t$, one has
$$\sigma((R, \Delta);\ba^t) \subseteq \sigma((R, \Delta'); \ib^s).$$
If $\ba \subseteq \ib \subseteq \overline{\ba}$, then 
$$\sigma((R, \Delta);\ba^t) = \sigma((R, \Delta); \ib^t).$$

\item $\sigma((R, \Delta);\ba^t)\mathfrak{a}  \subseteq \sigma((R, \Delta);\ba^{t+1})$.

\item
Let $W$ be a multiplicatively closed subset of $R$, and let $\Delta_W$ and $\ba_W$ be the images of $\Delta$ and $\ba$ in $R_W$, respectively. Then
$$\sigma((R_W, \Delta_W); \ba_W^t)=\sigma((R, \Delta); \ba^t)R_W.$$

\item
If $R$ is locally an integral domain and if the non-strongly-F-regular locus of $\Spec R$ is contained in $V(\ba)$,
then for any $\epsilon >0$,
$$\sigma((R, \Delta);\ba^t) \subseteq \tau_b((R, (1-\epsilon)\Delta); \ba^{t-\epsilon}).$$
\item
$(R, \Delta, \overline{\ba^{\lceil t \bullet \rceil}})$ is sharply F-pure if and only if $\sigma((R,\Delta);\ba^t)=R$.

\end{enumerate}
\end{prop}
\begin{proof}
(1) It is obvious, also see \cite[Proposition 3.2]{Bl}.

(2) By the definition of $\sigma_1((R, \Delta); \ba^t)$, we have
\begin{align*}
\sigma_1((R, \Delta);\ba^t)\mathfrak{a}&=\sum_{e \ge 1}\sum_{\varphi} \varphi(F^e_*(\overline{\ba^{\lceil t (p^e-1) \rceil}}\mathfrak{a}^{[p^e]}))\\
&\subseteq \sum_{e \ge 1}\sum_{\varphi} \varphi(F^e_*\overline{\ba^{\lceil (t+1) (p^e-1) \rceil}})\\
&= \sigma_1((R, \Delta);\ba^{t+1}).
\end{align*}
Applying this inclusion to the definition of $\sigma_2((R, \Delta);\ba^t)$, we have
\begin{align*}
\sigma_2((R, \Delta);\ba^t)\mathfrak{a}&=\sum_{e \ge 1}\sum_{\varphi} \varphi(F^e_*(\sigma_1((R, \Delta);\ba^t)\overline{\ba^{\lceil t (p^e-1) \rceil}}\mathfrak{a}^{[p^e]}))\\
&\subseteq \sum_{e \ge 1}\sum_{\varphi} \varphi(F^e_*(\sigma_1((R, \Delta);\ba^{t+1}) \overline{\ba^{\lceil (t+1) (p^e-1) \rceil}}))\\
&=\sigma_2((R, \Delta);\ba^{t+1}).
\end{align*}
Inductively we have $\sigma_n((R, \Delta);\ba^t)\mathfrak{a} \subseteq \sigma_n((R, \Delta);\ba^{t+1})$ for all integers $n \ge 1$, so that $\sigma((R, \Delta);\ba^t)\mathfrak{a} \subseteq \sigma((R, \Delta);\ba^{t+1})$.

(3)
It is immediate from \cite[Lemma 2.18]{Bl}.

(4)
We may assume that $R$ is a local domain and set $X = \Spec R$.  
It follows from \cite[p.63, Theorem]{Ho} that there exists an integer $m \ge 1$ such that every nonzero element of $\ba^m$ is a big test element for $R$.
Then $\tau_b((R, (1-\epsilon)\Delta); \ba^{t-\epsilon})$ is equal to $\sum_{e \ge n}\sum_{\psi}\psi(F^e_*\ba^{\lceil (t-\epsilon)p^e \rceil+m})$, where $n$ is an arbitrary positive integer and $\psi$ ranges over all the elements of $\Hom_R(F^e_*R(\lceil p^e(1-\epsilon)\Delta \rceil), R)$. 
By \cite{Hu}, there exists a positive integer $k$ such that $\overline{\ba^{n+k}} \subseteq \ba^n$ for all $n \ge 0$. 
We take a sufficiently large $n$ such that for all $e \geq n$, $\lceil t(p^e-1) \rceil \ge \lceil (t-\epsilon)p^e \rceil+m+k$, $\lceil (p^e-1) \Delta \rceil \ge \lceil p^e(1-\epsilon) \Delta \rceil$ and that $\sigma((R, \Delta);\ba^t)$ is contained in $\sum_{e \ge n}\sum_{\varphi}\varphi(F^e_*\overline{\ba^{\lceil t (p^e-1) \rceil}})$,  where $\varphi$ ranges over all the elements of $\Hom_R(F^e_*R(\lceil (p^e-1)\Delta \rceil), R)$.
Then
\begin{align*}
\sigma((R, \Delta);\ba^t) &\subseteq \sum_{e \ge n}\sum_{\varphi}\varphi(F^e_*\overline{\ba^{\lceil t (p^e-1) \rceil}})\\
&\subseteq \sum_{e \ge n}\sum_{\psi}\psi(F^e_*\overline{\ba^{\lceil (t-\epsilon) p^e \rceil+m+k}})\\
&\subseteq \sum_{e \ge n}\sum_{\psi}\psi(F^e_*\ba^{\lceil (t-\epsilon) p^e \rceil+m})\\
&=\tau_b((R, (1-\epsilon)\Delta); \ba^{t-\epsilon}).
\end{align*}

(5)
By the definition of sharp F-purity, $(R, \Delta, \overline{\ba^{\lceil t \bullet \rceil}})$ is sharply F-pure if and only if $\sigma_1((R,\Delta);\ba^t)=R$.
However,
$\sigma_1((R,\Delta);\ba^t)=R$ if and only if $\sigma_n((R,\Delta);\ba^t)=R$ for all  integers $n \ge 1$.
\end{proof}

\begin{rem}[Compare with \ref{maxnonlcdefsm}]
Even if $R$ is regular, the equality
$$\sigma((R, \Delta);\ba^t)= \tau_b((R, (1-\epsilon)\Delta); \ba^{t-\epsilon})$$
does not hold for any $\epsilon>0$ in general.
We give two easy examples.
\begin{enumerate}
\item
Let $R=k[x]$ be the one-dimensional polynomial ring over an $F$-finite field $k$ of characteristic $p>0$ and let $\ba=(x^p)$.
Then $\sigma(R, \ba^{\frac{1}{p}})=x$, while $\tau_b(R, \ba^{\frac{1}{p}-\epsilon})=R$ for any $\epsilon>0$.

\item Let $R=k[x, y]$ be the two-dimensional polynomial ring over an $F$-finite field $k$ of characteristic $p>0$ and let $\Delta=\Div(x^3-y^2)$.
If $p \equiv 2\text{ mod }3$, then $\sigma(R, \frac{5p-1}{6p}\Delta)=(x, y)$, while $\tau_b(R, (\frac{5p-1}{6p}-\epsilon)\Delta)=R$ for any $\epsilon>0$.
\end{enumerate}
The first example is also a counterexample to Theorem \ref{monomialcharp} below when the denominator of $t$ is divisible by $p$.

\end{rem}

\section{Non-F-pure ideals vs. non-lc ideal sheaves}
\label{secNonFpureVsNonLC}
In this section, we explore the relationship between non-F-pure ideals and non-lc ideal sheaves.

In Theorem \ref{monomialchar0}, we gave a combinatorial description of the non-lc ideal sheaf $\mJ'(X, \ba^t)$ associated to a monomial ideal $\ba$ on $X=\C^n$.
We show that the non-F-pure ideal $\sigma(R, \ba^t)$ has a similar description when $\ba$ is a monomial ideal of the polynomial ring $R=k[x_1, \dots, x_n]$ over an $F$-finite field $k$.

\begin{thm}[Compare with Theorem \ref{monomialchar0}]\label{monomialcharp}
Let $\ba$ be a monomial ideal of the polynomial ring $R:=k[x_1, \dots, x_n]$ over an $F$-finite field $k$ of characteristic $p>0$.
Let $t>0$ be a rational number whose denominator is not divisible by $p$.
Then the non-F-pure ideal $\sigma(R, \ba^t)$ is the monomial ideal generated by all monomials $x^v$ whose exponent vectors satisfy the condition that
$$v+\mathbf{1} \in P(t \cdot \ba),$$
where $P(t \cdot \ba)$ is the Newton polyhedron of $t \cdot \ba$.
\end{thm}

\begin{proof}
We denote by $I(R, \ba^t)$ the monomial ideal generated by all monomials $x^v$ whose exponent vectors satisfy the condition that $v+\mathbf{1} \in P(t \cdot \ba)$.
It follows from \cite[Theorem 4.8]{HY} that for sufficiently small $1 \gg \epsilon>0$, the generalized test ideal $\tau_b(R, \ba^{t-\epsilon})$ coincides with $I(R, \ba^t)$.
By Proposition \ref{basicnonFpureideal} (4), $\sigma(R, \ba^t)$ is contained in $I(R, \ba^t)$.

Hence, we will prove the converse inclusion.
For each integer $e \ge 1$, let $\phi^e:F^e_*R \to R$ be the $R$-linear map
such that
$$
\phi^e(x_1^{l_1} x_2^{l_2} \dots x_n^{l_n}) = \left\{ \begin{array}{l l} 1 & \text{if } l_1 = l_2 = \ldots = l_n = p^e - 1 \\ 0 & \text{whenever $l_i \leq p^e - 1$ for all $i$ and $l_i < p^e - 1$ for some $i$.} \end{array} \right.
$$
Then the ideal $\phi^e(F^e_*\overline{\ba^{\lceil t(p^e-1) \rceil}})$
is generated by monomials, because everything involved is $\Z^n$-graded.
The monomial $x^v$ is in the ideal $\phi^e(F^e_*\overline{\ba^{\lceil t(p^e-1) \rceil}})$ if and only if
\begin{equation}
\label{monomialproof1}
p^ev \in P(\lceil t(p^e-1) \rceil \cdot \ba)-(p^e-1) \mathbf{1}.
\end{equation}
Since the denominator of $t$ is not divisible by $p$, there are infinitely many $e \in \N$ such that $t(p^e-1)$ is an integer.
For such $e$, dividing out by $p^e-1$, we can rephrase $(\ref{monomialproof1})$ into the condition that $\frac{p^e}{p^e-1}v+\mathbf{1} \in P(t \cdot \ba)$.
By taking a sufficiently large $e$, this is equivalent to saying that $v+\mathbf{1} \in P(t \cdot \ba)$.
Thus, by the definition of $\sigma_1(R, \ba^t)$, the ideal $I(R, \ba^t)$ is contained in $\sigma_1(R, \ba^t)$.

Similarly, for each $e \ge 1$, the monomial $x^v$ is in $\phi^e(F^e_*(\sigma_1(R, \ba^t)\overline{\ba^{\lceil t(p^e-1) \rceil}}))$ if it is in $\phi^e(F^e_*(I(R, \ba^t)\overline{\ba^{\lceil t(p^e-1) \rceil}}))$ which happens if and only if
\begin{equation}
\label{monomialproof2}
p^ev \in (P(t \cdot \ba)-\mathbf{1})\cap \Z_{\ge 0}^{n}+P(\lceil t(p^e-1) \rceil \cdot \ba)-(p^e-1) \mathbf{1}.
\end{equation}
We will show that for all sufficiently large $e$ such that also $t(p^e-1)$ is an integer, Equation $(\ref{monomialproof2})$ is equivalent to the condition that $v+\mathbf{1} \in P(t \cdot \ba)$, that is, $x^v \in I(R, \ba^t)$. 
First suppose that $(\ref{monomialproof2})$ holds for such $e$. In particular, 
$$p^ev \in (P(t \cdot \ba)-\mathbf{1})+P(t(p^e-1) \cdot \ba)-(p^e-1)\mathbf{1}
=  p^e P(t \cdot \ba)-p^e \mathbf{1}.$$
Dividing out by $p^e$, we see that $v+\mathbf{1} \in P(t \cdot \ba)$. 
Conversely, suppose that $v+\mathbf{1} \in P(t \cdot \ba)$. 
Since $v$ is in $\Z^n_{\ge 0}$, this can be rephrased to say that $v \in (P(t \cdot \ba)-\mathbf{1})\cap \Z_{\ge 0}^{n}$. 
Multiplying both sides by $p^e-1$, we have that  
$$(p^e-1)v \in (p^e-1) (P(t \cdot \ba)-\mathbf{1})=P(\lceil t(p^e-1)\rceil \cdot \ba)-(p^e-1)\mathbf{1}$$
for all $e$ such that $t(p^e-1)$ is an integer. Finally,  we see that for such $e$, 
\begin{align*}
p^e v&=v+(p^e-1)v\\
&=(P(t \cdot \ba)-\mathbf{1})\cap \Z_{\ge 0}^{n}+P(\lceil t(p^e-1) \rceil \cdot \ba)-(p^e-1) \mathbf{1}.
\end{align*}

So, what it comes down to is that 
$$I(R, \ba^t) \subseteq \phi^e(F^e_*(\sigma_1(R, \ba^t)\overline{\ba^{\lceil t(p^e-1) \rceil}})) \subseteq \sigma_2(R, \ba^t).$$
Inductively we conclude that $I(R, \ba^t)$ is contained in $\sigma(R, \ba^t)$, which completes the proof of Theorem \ref{monomialcharp}.
\end{proof}

\begin{thm}\label{Fpure=>lc}
Let $R$ be an $F$-finite normal ring of characteristic $p>0$ and $\Delta$ be an effective $\R$-divisor on $X:=\Spec R$ such that $K_X+\Delta$ is $\R$-Cartier.
Let $\ba \subseteq R$ be a nonzero ideal and $t >0$ be a real number.
If $f:\widetilde{X} \to X$ is a proper birational morphism from a normal scheme $\widetilde{X}$ such that $\ba \cO_{\widetilde{X}}=\cO_{\widetilde{X}}(-Z)$ is invertible and $K_{\widetilde{X}}+\Delta_{\widetilde{X}}=f^*(K_X+\Delta)+tZ$, then one has an inclusion
$$\sigma((R, \Delta);\ba^t) \subseteq H^0(\widetilde{X}, \cO_{\widetilde{X}}(\lceil K_{\widetilde{X}}- f^*(K_X+\Delta)-tZ+\varepsilon F \rceil))$$
for sufficiently small $0 \le \varepsilon \ll 1$, where $F=\mathrm{Supp}\; \Delta_{\widetilde{X}}^{\ge 1}$.
\end{thm}
\begin{proof}
The proof is very similar to those of \cite[Theorem 3.3]{HW} and \cite[Theorem 2.13]{Ta}.

We may assume that $R$ is local.
Let $c \in \sigma((R,\Delta);\ba^t)$.
By Lemma \ref{nonFpurecontainment}, we may assume that there exist a sufficiently large $q=p^e$, a nonzero element $a \in \overline{\ba^{\lceil t(q-1) \rceil}}$ and an $R$-linear map $\varphi:F^e_*R(\lceil (q-1)\Delta \rceil) \to R$ sending $a$ to $c$.
The map $\varphi$ induces an $R$-linear map $\varphi':F^e_*R(\lceil (q-1)\Delta \rceil+\Div_X(a)) \to R(\Div_X(c))$ sending $1$ to $1$.
Since
\begin{align*}
& \sHom_R(F^e_*R(\lceil (q-1)\Delta \rceil+\Div_X(a)), R(\Div_X(c))) \\
\cong & F^e_*\cO_X((1-q)K_X+q \Div_X(c)-\lceil (q-1)\Delta \rceil-\Div_X(a))
\end{align*}
by Grothendieck duality,
we can regard $\varphi'$ as a rational section of $\cO_{\widetilde{X}}((1-q)K_{\widetilde{X}}+q f^{-1}_*\Div_X(c))$.
Let $D$ be the divisor on $\widetilde{X}$ corresponding to $\varphi'$.
Then $D$ is linearly equivalent to $(1-q)K_{\widetilde{X}}+q f^{-1}_*\Div_X(c)$ and $f_*D \geq \lceil (q-1)\Delta \rceil +\Div_{X}(a)$.

Set $Y:=\widetilde{X} \setminus \mathrm{Supp}\; D^{<0}$.
Since $\mathrm{Supp}\; D^{<0}$ is supported on the exceptional locus, $\varphi'$ lies in the global section of 
$$F^e_*\cO_Y((1-q)K_Y+q f^{-1}_*\Div_X(c)|_Y) \cong F^e_*\sHom_{\cO_Y}(F^e_*\cO_Y, \cO_Y(f^{-1}_*\Div_X(c)|_Y)).$$
We will prove that the coefficient of $D-qf^{-1}_*\Div_X(c)$ in each irreducible component is less than or equal to $q-1$.
Assume to the contrary that there exists an irreducible component $D_0$ of $D$ whose coefficient  is greater than or equal to $q(\mathrm{ord}_{D_0}(f^{-1}_*\Div_X(c))+1)$, where $\mathrm{ord}_{D_0}$ denotes the order along $D_0$.
Note that $D_0$ intersects $Y$.
Set $B:=f^{-1}_*\Div_X(c)-\mathrm{ord}_{D_0}(f^{-1}_*\Div_X(c))D_0$.
Then $\varphi'$ lies in the global section of 
$$F^e_*\cO_Y((1-q)K_Y+qB|_Y-qD_0|_Y) \cong F^e_*\sHom_{\cO_Y}(F^e_*\cO_Y(qD_0|_Y), \cO_Y(B|_Y)).$$
Combining $\varphi'$ with the natural inclusion map $\cO_Y(D_0|_Y) \hookrightarrow F^e_*\cO_Y(qD_0|_Y)$,  we obtain the $\cO_Y$-linear map $\psi:\cO_Y(D_0|_Y) \to \cO_Y(B|_Y)$.
Since $\varphi'$ sends $1$ to $1$, $\psi$ also sends $1$ to $1$.
However, since $D_0$ is not contained in $\Supp B$, this is a contradiction.
Thus, every coefficient of $D-qf^{-1}_*\Div_X(c)$ is less than or equal to $q-1$.

Now we put $G:=\frac{1}{q-1}D-f^{-1}_*\Delta$.
Then $G$ is $\Q$-linearly equivalent to $-(K_{\widetilde{X}}+f^{-1}_*\Delta)+\frac{q}{q-1}f^{-1}_*\Div_X(c)$, so that $f_*G$ is $\Q$-linearly equivalent to $-(K_X+\Delta)+\frac{q}{q-1}\Div_X(c)$.
There are finitely many prime divisors $E_j$ on $\widetilde{X}$ such that
$$K_{\widetilde{X}} \underset{\text{$\R$-lin.}}{\sim} f^*(K_X+\Delta)+tZ+\sum_j a_jE_j$$
where $a_j$ are real numbers chosen as $f^{-1}_*\Delta+tZ+\sum_j a_jE_j$ is $f$-exceptional.
Hence, one has
$$G-f^*f_*G+f^{-1}_*\Delta+tZ+\frac{q}{q-1}\left(\Div_{\widetilde{X}}(c)-f^{-1}_*\Div_X(c)\right)+\sum_j a_jE_j=0,$$
because it is an $f$-exceptional divisor $\R$-linearly equivalent to zero.
On the other hand,
$$f_*G \ge \frac{1}{q-1}\lceil (q-1)\Delta \rceil+\frac{1}{q-1}\Div_X(a)- \Delta \ge \frac{1}{q-1}\Div_X(a).$$
Putting it all together, we obtain
\begin{align*}
-a_j=&\mathrm{ord}_{E_j}\left(G-f^*f_*G+f^{-1}_*\Delta+tZ+\frac{q}{q-1}\left(\Div_{\widetilde{X}}(c)-f^{-1}_*\Div_X(c)\right)\right) \\
\le & \frac{1}{q-1}\mathrm{ord}_{E_j}\left(D-\Div_{\widetilde{X}}(a)+t(q-1)Z-qf^{-1}_*\Div_X(c)\right)+\frac{q}{q-1}\nu_{E_j}(c)\\
\le & \frac{1}{q-1}\mathrm{ord}_{E_j}\left(D-qf^{-1}_*\Div_X(c)\right)+\frac{q}{q-1}\nu_{E_j}(c)\\
 \le & 1+\frac{q}{q-1}\nu_{E_j}(c),
\end{align*}
where $\mathrm{ord}_{E_j}$ denotes the order along $E_j$ and $\nu_{E_j}$ is the valuation corresponding to $E_j$.
Since $q$ is sufficiently large, we conclude that  $a_j+ v_{E_j}(c)  \ge -1$ for all $j$, which implies that $c$ lies in $H^0(\widetilde{X}, \cO_{\widetilde{X}}(\lceil K_{\widetilde{X}}- f^*(K_X+\Delta)-tZ  +\varepsilon F \rceil))$ for sufficiently small $0 \le \varepsilon \ll 1$.
\end{proof}

\begin{conj}
Let $R$ be a normal ring essentially of finite type over a field of characteristic zero, and let $\Delta$ be an effective $\Q$-divisor on $X:=\Spec R$ such that $K_X+\Delta$ is $\Q$-Cartier.
Let $\ba \subseteq R$ be a nonzero ideal and $t >0$ be a real number.
Denoting by $(R_p, \Delta_p, \ba_p)$ the reduction modulo $p$ of the triple $(R, \Delta, \ba)$ and by $\mJ'((X, \Delta); \ba^t)_p$ that of the maximal non-lc ideal $\mJ'((X, \Delta); \ba^t)$,
one has
$$\mJ'((X, \Delta); \ba^t)_p=\sigma((R_p, \Delta_p); \ba_p^t)$$
for infinitely many primes $p$.
\end{conj}

\begin{rem}
It follows from Theorem \ref{Fpure=>lc} that $\sigma((R_p, \Delta_p); \ba_p^t)$ is contained in $\mJ'((X, \Delta); \ba^t)_p$ for all sufficiently large primes $p$.
However, the converse inclusion does not hold for all sufficiently large primes $p$ in general.
For example, let $E \subseteq \mathbb{P}^2_{\Q}$ be an elliptic curve over the rational numbers and $X=\Spec R$ be the affine cone over $E$.
Since $X$ has only log canonical singularities, $\mJ'(X, 0)=R$.
On the other hand, $\sigma(R_p, 0)=R_p$ if and only if $p$ is not supersingular prime for $E$.
It is known by Elkies \cite{El} that there are infinitely many supersingular primes for $E$.
Hence, it cannot happen that $\mJ'(X, 0)_p=\sigma(R_p, 0)$ for all sufficiently large primes $p$.
The reader is referred to \cite[Example 4.6]{MTW} for a more detailed explanation.
\end{rem}

\section{The restriction theorem for non-F-pure ideals}\label{secResThmForNonFpure}

In this section, we formulate the restriction theorem for non-F-pure ideals when $\ba$ is the unit ideal.

For simplicity, we may assume that $R$ is an $F$-finite normal local ring of characteristic $p>0$, and set $X=\Spec R$.
Then there exists a bijection of sets:
\begin{equation*}
\small
\left\{ \begin{matrix}\text{Effective $\Q$-divisors $\Delta$ on $X$ such }\\\text{that $(p^e - 1)(K_X + \Delta)$ is Cartier}\end{matrix} \right\} \leftrightarrow \left\{ \text{Nonzero elements of $\Hom_{R}(F^e_* R, R)$} \right\} \Big/ \sim
\end{equation*}
where the equivalence relation on the right hand side identifies two maps $\phi_1, \phi_2 \in \Hom_{R}(F^e_* R, R)$ if there exists some unit $u \in R$ such that $\phi_1(x) = \phi_2(ux)$.
The reader is referred to \cite{Sc2} for the details of this correspondence.

Given a map $\phi \in \Hom_{R}(F^e_* R, R)$ and an integer $l \ge 1$, the $l^{\rm th}$ iteration $\phi^l$ of $\phi$ is defined as follows:
$$\phi^l=\phi \circ (F^e_*\phi) \circ \dots \circ (F^{(l-1)e}_*\phi) \in  \Hom_{R}(F^{le}_* R, R).$$
We remark that if $\phi$ corresponds to some effective $\Q$-divisor $\Delta$ on $X$ such that $(p^e - 1)(K_X + \Delta)$ is Cartier, then $\phi^l$ corresponds to the same divisor $\Delta$ for every $l \in \N$.
This is equivalent to saying that if $\Hom_{R}(F^{e}_* R((p^e-1)\Delta), R)$ is a free $F^e_*R$-module generated by $\phi$, then $\Hom_{R}(F^{le}_* R((p^{le}-1) \Delta), R)$ is a free $F^{le}_*R$-module generated by $\phi^l$ for every $l \in \N$.

\begin{lem}\label{QCartiernonFpureideal}
Let $(R, \m)$ be an $F$-finite normal local ring of characteristic $p>0$ and $\Delta$ be an effective $\Q$-divisor on $X:=\Spec R$ such that $(p^{e_0}-1)(K_X+\Delta)$ is Cartier for some $e_0 \in \N$.
Let $\phi_{e_0}:F^{e_0}R \to R$ be the $R$-linear map corresponding to $\Delta$.
Then for all sufficiently large $l \in \N$, one has
$$\sigma(R, \Delta)=\phi_{e_0}^l(F^{le_0}_*R).$$
\end{lem}

\begin{proof}
Since $\phi_{e_0}(F^{e_0}_*R) \subseteq \sigma_1(R, \Delta)$, we have $\phi_{e_0}^2(F^{2e_0}_*R) \subseteq \sigma_2(R, \Delta)$ by the definition of $\sigma_2(R, \Delta)$.
Inductively we have $\phi_{e_0}^n(F^{ne_0}_*R) \subseteq \sigma_n(R, \Delta)$ for all $n \in \N$.
It then follows from Lemma \ref{nonFpurecontainment} that for all sufficiently large $l \in \N$,
$$\phi_{e_0}^l(F^{le_0}_*R) \subseteq \sigma(R, \Delta) \subseteq \sum_{e \ge le_0}\sum_{\varphi} \varphi(F^e_*R),$$
where $\varphi$ ranges over all elements of $\Hom_R(F^e_*R(\lceil (p^e-1)\Delta \rceil), R)$.
Hence, it suffices to show that $\sum_{e \ge le_0}\sum_{\varphi} \varphi(F^e_*R) \subseteq \phi_{e_0}^l(F^{le_0}_*R)$.
Let $\varphi \in \Hom_R(F^e_*R(\lceil (p^e-1) \Delta \rceil), R)$ with $e \ge le_0$.
Then by \cite[Corollary 3.10]{Sc2}, there exists an $R$-linear map $\psi_i: F^i_*R(\lceil (p^i-1) \Delta \rceil) \to R$ such that $\varphi=\phi^{l}_{e_0} \circ (F^{le_0}_*\psi_i)$ with $i=e-le_0$.
Thus, one has
$$\varphi(F^{e}_*R)=\phi^{l}_{e_0}((F^{le_0}_*\psi_i)(F^{e}_*R))\subseteq \phi^{l}_{e_0}(F^{le_0}_*R).$$
\end{proof}

\begin{rem}\label{computationFpure}
In the case of a regular ring $R$, we have the following description of non-F-pure ideals:
let $\Delta=t \cdot \Div(f)$ be an effective $\Q$-divisor on $\Spec R$ such that the denominator of $t$ is not divisible by $p$.
If $J \subseteq R$ is an ideal, one defines $J^{[1/p^e]}$ to be the smallest ideal $I$ such that $I^{[p^e]} \supseteq J$.
 Then by Lemma \ref{QCartiernonFpureideal} and \cite[Proposition 3.10]{BSTZ}, one has
 $$\sigma(R, \Delta)=(f^{t (p^e-1)})^{[1/p^e]}$$
for sufficiently large and divisible $e$ such that $t(p^e-1)$ is an integer.
For example, let $R=\F_p[x, y]$ be the two-dimensional polynomial ring over $\F_p$ and let $\Delta=\Div(x^3-y^2)$.
Note that $\{x^iy^j\}_{p^e-1\ge i, j \ge 0}$ is a basis of $R$ over $R^{p^e}$ for all $e \in \N$.
Then by \cite[Proposition 2.5]{BMS}, taking a sufficiently large $e$, one has
$$\sigma(R, \Delta)=\left((x^3-y^2)^{p^e-1}\right)^{[1/p^e]}=(x, y).$$
\end{rem}

\begin{propdef}[\textup{\cite[Theorem 5.2]{Sc2}}]\label{Fdifferent}
Let $R$ be an $F$-finite normal local ring of characteristic $p>0$ and $D+B$ be an effective $\Q$-divisor on $X:=\Spec R$ such that $D$ is a normal prime divisor with defining ideal $Q \subseteq R$ and that $D$ is not contained in $\Supp B$.
Assume that there exists $e \in \N$ such that $(p^{e}-1)(K_X+D+B)$ is Cartier.

Let $\phi: F^e_*R \to R$ be the $R$-linear map corresponding to $D+B$.
Since the localized ring $R_Q$ is a DVR, $\phi(F^e_*Q) \subseteq Q$ (that is, $Q$ is an F-pure center of $(R, D+B)$. See Definition \ref{Fpurecenterdef} for the definition of F-pure centers).
Then we have the following commutative diagram:
$$
\xymatrix{
 F^{e}_*R   \ar[r]^{\phi} \ar[d] & R \ar[d] \\
 F^{e}_*(R/Q)  \ar[r]^{\phi_Q} & R/Q,
}
$$
where the vertical maps are the natural surjections.
We denote by $B_{R/Q}$ the effective $\Q$-divisor on $D$ corresponding to $\phi_Q$.
It is easy to check the following properties:
\begin{enumerate}
\item[(i)]
$(p^e-1)(K_D+B_{R/Q})$ is a Cartier divisor.
\item[(i')]
$\Hom_{R/Q}(F^e_*(R/Q)((p^e-1)B_{R/Q}), R/Q)$ is a  free $F^e_*(R/Q)$-module generated by $\phi_Q$.
\item[(ii)]
$(R, D+B)$ is sharply F-pure if and only if $(R/Q, B_{R/Q})$ is sharply F-pure.
\end{enumerate}
\end{propdef}

\begin{rem}
The divisor $B_{R/Q}$ defined above is canonically determined and exists even outside the local setting.  Explicitly, if $X$ is an $F$-finite normal irreducible scheme and $D$ and $B$ are as above, then there exists a divisor $B_{D}$ on $D$ (replacing $B_{R/Q}$) satisfying the properties (i), (i') and (ii) above locally and also satisfying the condition that $(K_X + D + B)|_{D} \sim_{\Q} B_{D}$.  See \cite[Remark 9.5]{Sc2} for details.
\end{rem}

\begin{conj}[\textup{cf.~\cite[Remark 7.6]{Sc2}}]\label{Fdifferent vs. different}
Let the notation be the same as in Definition \ref{Fdifferent}.
Then $B_{R/Q}$ coincides with the different $B_D$ of $(X, D+B)$ on $D$ $($see Section \ref{sec-diff} for the definition of differents$)$.
\end{conj}

\begin{rem}
Conjecture \ref{Fdifferent vs. different} holds true if $D$ is Cartier in codimension two.
The reader is referred to \cite[Section 7]{Sc2} for details.
\end{rem}

Now we state our restriction theorem for non-F-pure ideals.
\begin{thm} \label{restrictionTheoremForNonFPure}
Let $R$ be an $F$-finite normal domain of characteristic $p>0$ and $D+B$ be an effective $\Q$-divisor on $X:=\Spec R$ such that $D$ is a normal prime divisor with defining ideal $Q \subseteq R$ and that $D$ is not contained in $\Supp B$.
Assume that $K_X+D+B$ is $\Q$-Cartier with index not divisible by $p$.
Then
$$\sigma(R,D+B)|_{D}=\sigma(R/Q, B_{R/Q}).$$
\end{thm}

\begin{proof}
The statement is local, so we may assume without loss of generality that $R$ is also local.
Since $K_X+D+B$ is $\Q$-Cartier with index not divisible by $p$, there exist infinitely many $e \in \N$ such that $(p^e-1)(K_X+D+B)$ is a Cartier divisor.
We fix one of such $e$.
Let $\phi \in \Hom_R(F^{e}_*R, R)$ be the $R$-linear map corresponding to $D+B$.
By the definition of $B_{R/Q}$, there exists $\phi_Q \in \Hom_{R/Q}(F^{e}_*(R/Q), R/Q)$ corresponding to $B_{R/Q}$
such that we have the following commutative diagram for each $l \in \N$ :
$$
\xymatrix{
F^{le}_*R   \ar[r]^{\phi^l} \ar[d] & R \ar[d] \\
F^{le}_*(R/Q)  \ar[r]^{\phi_Q^l} & R/Q,
}
$$
where the vertical maps are natural quotient maps.
Thus, it follows from Lemma \ref{QCartiernonFpureideal} that for a sufficiently large $l$,
$$\sigma(R, D+B)|_D=\phi^l(F^{le}_*R)R/Q=\phi_Q^l(F^{le}_*(R/Q))=\sigma(R/Q, B_{R/Q}).$$
\end{proof}

In fact, the previous restriction even holds when restricting to an $F$-pure center of arbitrary codimension.

\begin{defn}[\cite{Sc3}]\label{Fpurecenterdef}
Suppose that $(X, \Delta)$ is a pair such that $K_X + \Delta$ is $\Q$-Cartier with index not divisible by $p$.  We say that a subvariety $W \subseteq X$ is a \emph{center of sharp $F$-purity for $(X, \Delta)$} if, after localizing at each point $x \in X$, any (equivalently, some) map $\phi : F^e_* \cO_{X, x} \to \cO_{X, x}$ corresponding to $\Delta$ (as at the start of this section) satisfies the property that
\[
\phi( F^e_* \I_{W, x}) \subseteq \I_{W, x}.
\]
Here $\I_{W}$ is the ideal sheaf defining $W$ and $\I_{W, x}$ is its stalk at $x \in X$.
 We simply call it an F-pure center of $(R, \Delta)$ if the context is clear.
\end{defn}

Given a pair $(X := \Spec R, \Delta)$ and a normal $F$-pure center $W \subseteq X$ with defining ideal $Q \subseteq R$ such that $(X, \Delta)$ is sharply $F$-pure at $Q$, then there exists a canonically determined $\Q$-divisor $\Delta_{R/Q}$ on $W$ satisfying the properties (i), (i') and (ii) from Proposition-Definition \ref{Fdifferent}.  The proof (and reference) are the same.

\begin{thm}
Let $R$ be an $F$-finite normal ring of characteristic $p>0$ and $\Delta$ be an effective $\Q$-divisor on $X:=\Spec R$ such that $K_X + \Delta$ is $\Q$-Cartier with index not divisible by $p$.  Suppose that $W \subseteq X$ is an $F$-pure center of $(X, \Delta)$ and also that $(X, \Delta)$ is sharply $F$-pure at the generic point of $W$.  Let us use $Q$ to denote the ideal of $W$.
Then
$$\sigma(R,\Delta)|_{W}=\sigma(R/Q, \Delta_{R/Q}).$$
\end{thm}
\begin{proof}
The proof is the same as in Theorem \ref{restrictionTheoremForNonFPure}.
The assumption that $(X, \Delta)$ is sharply F-pure at the generic point of $W$ is needed to define the $\Q$-divisor $\Delta_{R/Q}$. 
\end{proof}

Compare the following example with Example \ref{restrictionex}.
\begin{ex}
Let $R=k[x,y]$ be the two-dimensional polynomial ring over an $F$-finite field $k$. Set $D=\Div(x)$ and $B=\Div(x^3-y^2)$. It then follows from Remark \ref{computationFpure} that
\begin{align*}
\sigma(R, D+B)&=(x^2, y),\\
\sigma(R/(x), B|_D)&=(y).
\end{align*}
Hence the restriction theorem holds in this case.
\end{ex}

\ifx\undefined\bysame
\newcommand{\bysame|{leavemode\hbox to3em{\hrulefill}\,}
\fi

\end{document}